\documentclass[11pt,reqno]{amsart}

\usepackage[T1]{fontenc}
\usepackage{lmodern}
\usepackage{microtype}
\usepackage{mathtools}
\usepackage{amssymb}
\usepackage{enumitem}
\usepackage[hidelinks]{hyperref}
\usepackage[nameinlink,noabbrev]{cleveref}

\numberwithin{equation}{section}

\newtheorem{theorem}{Theorem}[section]
\newtheorem{proposition}[theorem]{Proposition}
\newtheorem{lemma}[theorem]{Lemma}
\newtheorem{corollary}[theorem]{Corollary}
\newtheorem{conjecture}[theorem]{Conjecture}
\theoremstyle{remark}
\newtheorem{remark}[theorem]{Remark}

\newcommand{\N}{\mathbb N}
\newcommand{\R}{\mathbb R}
\newcommand{\e}{\mathrm e}
\newcommand{\ii}{\mathrm i}
\newcommand{\G}{\Gamma}

\title[The KKT conjecture and Laguerre dispersion]
{The Koornwinder--Kostenko--Teschl Conjecture for Jacobi Polynomials\\
and the Discrete Laguerre Phase Transition}

\author{Yu-Tian Li}
\address{School of Mathematics and Statistics, Nanfang College, Guangzhou, 510970, Guangdong,
China}
\email{yutianlee@gmail.com}
\date{}

\hypersetup{
  pdftitle={The Koornwinder-Kostenko-Teschl Conjecture for Jacobi Polynomials and the Discrete Laguerre Phase Transition},
  pdfauthor={Yu-Tian Li},
  pdfkeywords={Jacobi polynomials, Bernstein-type inequalities, Koornwinder-Kostenko-Teschl conjecture, canonical products, discrete Laguerre operator, dispersion estimates, phase transition}
}


\begin{document}

\begin{abstract}
We prove the refined Koornwinder--Kostenko--Teschl conjecture.  For the
normalized weighted Jacobi function and all
\(n\in\mathbb N_0\), \(\alpha,\beta\ge0\), and \(-1\le x\le1\),
\[
|g_n^{(\alpha,\beta)}(x)|
\le
\left[
\frac{(n+1)(n+\alpha+\beta+1)}
{(n+\alpha+1)(n+\beta+1)}
\right]^{1/4}\le1.
\]
The proof combines a central contour estimate with Sturm--Sonin
localization and an exact inverse moment.  It also yields an
\(n\)-uniform extreme-lobe theorem and a sharp canonical-product
first-lobe principle.

Applied to the discrete Laguerre evolution, the estimate gives the
optimal positive-parameter decay.  For \(-1<\alpha\le0\), a
complementary one-sided Jacobi inequality gives the exact norm
\(\|\e^{-\ii tH_\alpha}\|_{\ell^1\to\ell^\infty}
=(1+t^2)^{-(1+\alpha)/2}\).
Thus the large-time decay exponent is \(\min\{1,1+\alpha\}\) for every
\(\alpha>-1\).  Bessel, Laguerre, and Darboux scaling limits show that
the temporal and fixed-diagonal exponents are optimal.
\end{abstract}

\maketitle

\section{Introduction}

Bernstein's inequality for Legendre polynomials initiated the study of
weighted uniform bounds for classical orthogonal polynomials.  In the
Jacobi case, the central questions are uniformity in the degree and
parameters and control at the hard edges; see
\cite{AntonovKholshevnikov1981,Bernstein1931,ChowGatteschiWong1994,
ErdelyiMagnusNevai1994,Gautschi2010,HaagerupSchlichtkrull2014,
Krasikov2007,Krasikov2008,Lorch1984}.  These estimates also enter
harmonic analysis and spectral theory, including multiplier theorems
on the Grushin sphere
\cite{CasarinoCiattiMartini2019,CasarinoCiattiMartini2021} and in
weighted extremal problems \cite{ChristiansenRubin2025}.  Here the
same inequality also determines the dispersive decay of the discrete
Laguerre evolution.

Let \(P_n^{(\alpha,\beta)}\) be the Jacobi polynomial in the standard
normalization.  For \(n\in\N_0\), \(\alpha,\beta\ge0\), and
\(x\in[-1,1]\), define
\begin{equation}\label{eq:def-g}
\begin{aligned}
g_n^{(\alpha,\beta)}(x)={}&
\left[
\frac{\G(n+1)\G(n+\alpha+\beta+1)}
{\G(n+\alpha+1)\G(n+\beta+1)}
\right]^{1/2}\\[-2pt]
&\times
\left(\frac{1-x}{2}\right)^{\alpha/2}
\left(\frac{1+x}{2}\right)^{\beta/2}
P_n^{(\alpha,\beta)}(x).
\end{aligned}
\end{equation}
At \(x=\pm1\), powers of the form \(0^0\) in \eqref{eq:def-g} are
understood by continuous extension in \(x\) with the parameters fixed.
We shall also use the same formula for \(\alpha>-1\), \(\beta\ge0\),
and \(-1\le x<1\), where it is finite; \Cref{thm:main} itself concerns
the nonnegative parameter quadrant.

When \(\alpha,\beta\in\N_0\), the functions \eqref{eq:def-g} occur as
matrix coefficients of irreducible unitary representations of
\(\mathrm{SU}(2)\); see
\cite{HaagerupSchlichtkrull2014,KKT2018}.  Unitarity gives
\begin{equation}\label{eq:basic-intro}
\lvert g_n^{(\alpha,\beta)}(x)\rvert\le1
\end{equation}
on the integer-parameter lattice.  In the same regime, the contour
analysis of Haagerup and Schlichtkrull yields the refinement
\begin{equation}\label{eq:target}
\lvert g_n^{(\alpha,\beta)}(x)\rvert
\le T_{n,\alpha,\beta}:=
\left[
\frac{(n+1)(n+\alpha+\beta+1)}
{(n+\alpha+1)(n+\beta+1)}
\right]^{1/4}.
\end{equation}
See \cite[Equation~(20)]{HaagerupSchlichtkrull2014} and
\cite[Lemma~4.3]{KKT2018}.

Koornwinder, Kostenko, and Teschl conjectured that
\eqref{eq:basic-intro} and its refinement \eqref{eq:target} hold for
all real \(\alpha,\beta\ge0\)
\cite[Conjecture~6.1]{KKT2018}.  We prove the refined statement.

\begin{theorem}\label{thm:main}
For every \(n\in\N_0\), \(\alpha,\beta\ge0\), and \(x\in[-1,1]\),
inequality \eqref{eq:target} holds.
\end{theorem}

Since
\[
(n+\alpha+1)(n+\beta+1)
-(n+1)(n+\alpha+\beta+1)=\alpha\beta\ge0,
\]
\Cref{thm:main} implies \eqref{eq:basic-intro}.  The constant one is
best possible: equality occurs at \(x=1\) when \(\alpha=0\), and at
\(x=-1\) when \(\beta=0\).  We do not claim that
\(T_{n,\alpha,\beta}\) is optimal for every fixed positive parameter
triple.

The representation-theoretic proof is confined to integer parameters;
analytic continuation does not preserve a modulus inequality.
For arbitrary real parameters, Haagerup and Schlichtkrull proved
\begin{equation}\label{eq:HS-intro}
(1-x^2)^{1/4}\lvert g_n^{(\alpha,\beta)}(x)\rvert
\le \frac{C}{(2n+\alpha+\beta+1)^{1/4}},
\qquad C<12.
\end{equation}
This bound degenerates at the endpoints.  Their optimized contour
argument gives \eqref{eq:target} on the central interval
\[
\lvert x\rvert\le
\frac{\alpha+\beta}{\alpha+\beta+2n},
\]
but leaves the two endpoint caps uncontrolled; see
\cite[Section~5, Case~2]{HaagerupSchlichtkrull2014}.

Koornwinder, Kostenko, and Teschl discovered that this parameter gap is
also a dispersive one.  They expressed the kernel of the discrete
Laguerre evolution in terms of \eqref{eq:def-g}:
\begin{equation}\label{eq:kernel-intro}
\left|\e^{-\ii tH_\alpha}(n,m)\right|
=\frac1{\sqrt{1+t^2}}
\left|
g_n^{(\alpha,m-n)}
\left(\frac{t^2-1}{t^2+1}\right)
\right|,
\qquad m\ge n.
\end{equation}
Consequently, the integer-parameter estimate gives
\[
\left\|\e^{-\ii tH_\alpha}\right\|_{\ell^1\to\ell^\infty}
\le (1+t^2)^{-1/2}
\]
for \(\alpha\in\N_0\); the exact \(\alpha=0\) identity had been proved
earlier in \cite{KostenkoTeschl2016}.  For arbitrary real
\(\alpha\ge0\), the Haagerup--Schlichtkrull estimate instead yielded
only
\[
\left|\e^{-\ii tH_\alpha}(n,m)\right|
\le \frac{C|t|^{-1/2}}{(n+m+\alpha+1)^{1/4}},
\qquad t\ne0;
\]
see \cite[Theorem~6.3]{KKT2018}.  This is the operator-theoretic
meaning of the passage from integer to real Jacobi parameters.

For \(a>0\), put
\[
R_*(a)=\frac{a+2}{2(a+1)}
\left(\frac{a+2}{a}\right)^a,
\qquad R_*(0):=1.
\]
The next result complements the parameter-dependent factor in
\Cref{thm:main} by an \(n\)-uniform estimate and identifies where the
maximum occurs.

\begin{theorem}[Extreme-lobe dominance]
\label{thm:intro-extreme-lobes}
Let \(n\ge1\) and \(\alpha,\beta>0\).  If \(\mu_+\) and \(\mu_-\)
are the maxima of \(\lvert g_n^{(\alpha,\beta)}\rvert\) on the nodal
intervals adjacent to \(x=1\) and \(x=-1\), respectively, then
\begin{equation}\label{eq:intro-extreme-lobes}
\max_{-1\le x\le1}|g_n^{(\alpha,\beta)}(x)|
=\max\{\mu_+,\mu_-\}
<\max\{R_*(\alpha)^{-1/2},R_*(\beta)^{-1/2}\}<1.
\end{equation}
If one or both parameters vanish, only the corresponding numerical
bound is asserted: with \(R_*(0)=1\),
\[
\max_{-1\le x\le1}|g_n^{(\alpha,\beta)}(x)|
\le\max\{R_*(\alpha)^{-1/2},R_*(\beta)^{-1/2}\},
\qquad \alpha,\beta\ge0.
\]
Consequently,
\begin{equation}\label{eq:intro-combined-bound}
\max_{-1\le x\le1}|g_n^{(\alpha,\beta)}(x)|
\le \min\!\left\{
T_{n,\alpha,\beta},
\max\{R_*(\alpha)^{-1/2},R_*(\beta)^{-1/2}\}
\right\}.
\end{equation}
\end{theorem}

In particular, the unrefined Koornwinder--Kostenko--Teschl (KKT)
inequality for positive degree has a contour-free proof.  The two branches in
\eqref{eq:intro-combined-bound} are complementary: the \(R_*\)-branch
is uniform in the degree, while \(T_{n,\alpha,\beta}\) is stronger in
some large-parameter regimes.

\paragraph{The dispersive phase transition.}
A complementary one-sided Jacobi inequality for \(-1<a\le0\), proved
in \Cref{thm:negative-contractivity}, yields the second main result.

\begin{theorem}[Dispersive phase transition]
\label{thm:intro-phase-transition}
Let \(\alpha>-1\).
\begin{enumerate}[label=\textup{(\roman*)}]
\item If \(-1<\alpha\le0\), then, for every \(t\in\R\),
\begin{equation}\label{eq:intro-negative-norm}
\left\|\e^{-\ii tH_\alpha}\right\|_{\ell^1\to\ell^\infty}
=(1+t^2)^{-(1+\alpha)/2}.
\end{equation}
\item If \(\alpha\ge0\), then
\[
\left\|\e^{-\ii tH_\alpha}\right\|_{\ell^1\to\ell^\infty}
\le (1+t^2)^{-1/2}.
\]
\item In every case,
\begin{equation}\label{eq:intro-phase-law}
\left\|\e^{-\ii tH_\alpha}\right\|_{\ell^1\to\ell^\infty}
\asymp_\alpha |t|^{-\min\{1,1+\alpha\}},
\qquad |t|\to\infty.
\end{equation}
\end{enumerate}
\end{theorem}

Part~\textup{(i)} strengthens the
\(O(|t|^{-(1+\alpha)})\) prediction of
\cite[Remark~6.6]{KKT2018} to an exact identity, valid for every time
and throughout the full range \(-1<\alpha\le0\).

For \(\alpha=0\), the \((0,0)\) kernel entry attains the positive
upper bound.  For \(\alpha>0\), we do not determine the exact kernel
norm; the lower bounds proving \eqref{eq:intro-phase-law} arise from
two distinct hard-edge scalings.

\paragraph{Further consequences.}
The estimate in \Cref{thm:main} remains finite at the hard edges but
does not decay with the degree when the parameters are fixed.  In an
independent companion paper, Bai and Li
\cite[Theorem~1.1]{BaiLiEMNKrasikov2026} prove the degree--parameter
order conjectured by Krasikov in its big-\(O\) form, and hence the
Erd\'elyi--Magnus--Nevai (EMN) conjecture.  After conversion between the two
Jacobi normalizations, their theorem supplies a complementary bulk
branch.  Taking the minimum of the endpoint and bulk estimates gives
a two-scale envelope whose crossover, for fixed parameters, occurs on
the natural hard-edge scale \(1-x^2\asymp n^{-2}\).  On each fixed
diagonal of the discrete Laguerre kernel it improves the spatial order
from \((n+m)^{-1/4}\) to \((n+m)^{-1/2}\); Darboux asymptotics show
that the latter exponent is optimal.

Three scaling limits account for the lower bounds used in
\Cref{thm:intro-phase-transition}: Bessel functions arise on fixed
diagonals when \(n/|t|\) is fixed, Laguerre functions arise when the
index separation is of order \(1+t^2\), and a Darboux profile governs
the large-index limit at fixed nonzero time.  The squared moduli of the
entries in the lowest row form a negative-binomial distribution and
lead to a natural extremal conjecture for normalized Meixner
coefficients.  The dispersive estimates also give the corresponding
\(\ell^p\)-\(\ell^{p'}\), Strichartz, and small-data scattering bounds.

\paragraph{Outline of the proof.}
The Haagerup--Schlichtkrull contour estimate handles the central
interval.  On each endpoint cap a Sturm--Sonin function reduces the
problem to the first-lobe maximum \(A\).  Rodrigues' formula gives the
exact inverse moment
\[
\int_0^B\frac{H(u)^2}{u}\,du=\frac1a,
\]
and the canonical-product principle bounds it below by
\(A^2R_*(a)/a\).  The remaining comparison with
\(T_{n,a,b}\) reduces to an explicit rational inequality.  Reflection,
continuity, and a separate degree-zero argument complete the proof.
The same Sonin function orders all lobe maxima around the vertex of its
quadratic coefficient, proving extreme-lobe dominance.  A reflected
inverse moment and a second Sonin argument treat negative parameters.
Neither contractivity theorem uses asymptotics.

\paragraph{Organization.}
\Cref{sec:central,sec:sonin,sec:overflow,sec:first-lobe,%
sec:completion,sec:degree-zero} prove \Cref{thm:main}.
\Cref{sec:negative-jacobi} establishes the one-sided
negative-parameter theorem.  \Cref{sec:endpoint-bulk,sec:laguerre}
give the two-scale synthesis and the Laguerre--Bessel consequences.
\Cref{sec:discrete-laguerre} proves the kernel norm, scaling, and
sharpness results, and \Cref{sec:nonlinear} records the Strichartz and
small-data scattering applications.

\section{Preliminaries and the central interval}\label{sec:central}

We shall use the reflection identity
\begin{equation}\label{eq:reflection}
g_n^{(\alpha,\beta)}(-x)
=(-1)^n g_n^{(\beta,\alpha)}(x),
\end{equation}
which follows from
\(
P_n^{(\alpha,\beta)}(-x)=(-1)^nP_n^{(\beta,\alpha)}(x).
\)
The Jacobi polynomial has \(n\) simple zeros in \((-1,1)\), and
\[
P_n^{(\alpha,\beta)}(1)
=\frac{\G(n+\alpha+1)}{n!\,\G(\alpha+1)}.
\]
These standard facts and Rodrigues' formula may be found in
\cite{Szego1975}.

The only nonstandard estimate imported into the proof of
\Cref{thm:main} is the following central bound.

\begin{lemma}[Central-contour estimate]\label{lem:central}
Let \(n\ge1\), \(\alpha,\beta\ge0\), and \(\alpha+\beta>0\).  Set
\[
\sigma=\frac{\alpha+\beta}{\alpha+\beta+2n}.
\]
Then \eqref{eq:target} holds whenever \(\lvert x\rvert\le\sigma\).
\end{lemma}

\begin{proof}
We spell out the real-parameter part of the contour argument, since it
is the only external analytic estimate used in the proof.  Put
\[
a=\frac{\alpha}{n},\qquad b=\frac{\beta}{n},\qquad
\rho=\left(\frac{1-x^2}{a+b+1}\right)^{1/2}.
\]
The Rodrigues--Cauchy representation
\cite[Equations~(6)--(7)]{HaagerupSchlichtkrull2014} is
\begin{equation}\label{eq:central-cauchy}
\begin{aligned}
(1-x)^\alpha(1+x)^\beta P_n^{(\alpha,\beta)}(x)
&=\left(-\frac12\right)^n I_n^{(\alpha,\beta)}(x),\\
I_n^{(\alpha,\beta)}(x)
&=\frac1{2\pi\ii}\int_\gamma
(1-z)^{n+\alpha}(1+z)^{n+\beta}
\frac{dz}{(z-x)^{n+1}}.
\end{aligned}
\end{equation}
For real parameters the principal branches in
\eqref{eq:central-cauchy} are analytic inside the circle
\(\gamma=C(x,\rho)\) provided that the circle misses both cuts.  The
two elementary equivalences
\[
x+\rho<1\iff x<\frac{a+b}{a+b+2},
\qquad
x-\rho>-1\iff x>-\frac{a+b}{a+b+2}
\]
show that this is precisely the case when \(|x|<\sigma\).

On writing \(z=x+\rho e^{\ii\theta}\), the modulus of the integrand
in \eqref{eq:central-cauchy} is \(e^{nf(\cos\theta)}\), where
\begin{align*}
f(t)={}&\frac{a+1}{2}
 \log\!\left(\rho^2+(1-x)^2-2\rho(1-x)t\right)\\
&+\frac{b+1}{2}
 \log\!\left(\rho^2+(1+x)^2+2\rho(1+x)t\right)-\log\rho.
\end{align*}
Indeed, for the principal branches
\(|(1-z)^{n+\alpha}|=|1-z|^{n+\alpha}\), and similarly at
\(-1\).  Set
\[
t_2=\frac{\rho^2+(1-x)^2}{2\rho(1-x)},
\qquad
t_1=-\frac{\rho^2+(1+x)^2}{2\rho(1+x)}.
\]
Then \([ -1,1]\subset[t_1,t_2]\), and, up to a constant independent
of \(t\),
\[
f(t)=\frac{a+1}{2}\log(t_2-t)
     +\frac{b+1}{2}\log(t-t_1).
\]
Thus \(f''<0\) on \((t_1,t_2)\), and its maximum occurs at
\[
t_0=\frac{(a+1)t_1+(b+1)t_2}{a+b+2}.
\]
Consequently \(|I_n^{(\alpha,\beta)}(x)|\le e^{nf(t_0)}\).
The direct evaluation of \(f(t_0)\), as in
\cite[Equations~(18)--(19)]{HaagerupSchlichtkrull2014}, together with
the real-parameter gamma inequality
\begin{align*}
&\frac{\G(n+1)\G(n+\alpha+\beta+1)}
{\G(n+\alpha+1)\G(n+\beta+1)}
\left[
\frac{(a+1)^{a+1}(b+1)^{b+1}}
{(a+b+1)^{a+b+1}}
\right]^n\\
&\hspace{35mm}\le
\left[
\frac{(n+1)(n+\alpha+\beta+1)}
{(n+\alpha+1)(n+\beta+1)}
\right]^{1/2},
\end{align*}
proved in \cite[Lemma~4.1]{HaagerupSchlichtkrull2014}, gives after
substitution in \eqref{eq:def-g}
\[
|g_n^{(\alpha,\beta)}(x)|\le T_{n,\alpha,\beta}.
\]
No multiplicative constant is lost.  This also explains directly why
the sentence in \cite[Section~5, Case~2]{HaagerupSchlichtkrull2014}
extends the integral-parameter calculation without modification.
The two points \(x=\pm\sigma\) follow by continuity.  The excluded
corner \((\alpha,\beta)=(0,0)\), for which the open interval is empty,
will be treated in \Cref{sec:completion}.
\end{proof}

\section{Sturm--Sonin localization on an endpoint cap}\label{sec:sonin}

The central estimate leaves two endpoint caps.  We treat the right cap
first; the left cap will follow by reflection when the proof is
assembled.

Throughout this section, assume \(n\ge1\), \(a>0\), and \(b\ge0\).
Set
\begin{equation}\label{eq:B-H-p}
B=n+a+b+1,\qquad
H(u)=g_n^{(a,b)}\left(1-\frac{2u}{B}\right),\qquad
p(u)=u\left(1-\frac uB\right).
\end{equation}

\begin{lemma}[Endpoint equation]\label{lem:endpoint-equation}
The function \(H\) satisfies
\begin{equation}\label{eq:sturm}
(pH')'+QH=0,
\end{equation}
where
\begin{equation}\label{eq:Q}
Q(u)=n+\frac12-\frac{n+1}{2B}
-\frac{\left(\left(1-\frac{n+1}{B}\right)u-a\right)^2}
{4u(1-u/B)}.
\end{equation}
Writing \(P=pQ\), \(r=1-(n+1)/B\), and \(\nu=n+r/2\), one has
\begin{equation}\label{eq:P-quadratic}
P(u)=-\frac{a^2}{4}
+\left(\nu+\frac{ra}{2}\right)u
-\left(\frac{\nu}{B}+\frac{r^2}{4}\right)u^2.
\end{equation}
\end{lemma}

\begin{proof}
Let \(y=P_n^{(a,b)}\), \(A(x)=1-x^2\), \(s=a+b\), and \(d=b-a\).
Conjugating the standard Jacobi equation by the square root of its
weight gives
\begin{equation}\label{eq:weighted-jacobi}
(Ag')'+\frac{F(x)}{A(x)}g=0,
\qquad
F(x)=\kappa A(x)-\frac14(d-sx)^2,
\end{equation}
for any constant multiple \(g\) of
\((1-x)^{a/2}(1+x)^{b/2}y(x)\), where
\[
\kappa=n(n+s+1)+\frac{s}{2}.
\]
Under \(x=1-2u/B\),
\[
A(x)=\frac{4u(B-u)}{B^2},
\qquad
d-sx=2\left(\frac{su}{B}-a\right),
\]
and
\[
(Ag')'=(u(B-u)H')'.
\]
Dividing the resulting equation by \(B\), using
\[
\frac{s}{B}=1-\frac{n+1}{B},
\qquad
\frac{\kappa}{B}
=n+\frac12-\frac{n+1}{2B},
\]
gives \eqref{eq:sturm}--\eqref{eq:Q}.  Multiplication by \(p\) and
expansion give \eqref{eq:P-quadratic}.
\end{proof}

Let \(s=a+b\) and
\(\sigma=s/(s+2n)\).  The point \(x=\sigma\) corresponds to
\begin{equation}\label{eq:u-sigma}
u_\sigma=\frac{B(1-\sigma)}2
=\frac{nB}{B+n-1}\le n.
\end{equation}

\begin{lemma}[Monotonicity of the Sturm product]\label{lem:P-monotone}
On \(0\le u\le u_\sigma\),
\[
P'(u)>0.
\]
More precisely,
\begin{equation}\label{eq:P-prime-boundary}
P'(u_\sigma)=\frac{(a+b)(n+a+1)}{2B}>0.
\end{equation}
\end{lemma}

\begin{proof}
Equation \eqref{eq:P-quadratic} shows that \(P''<0\), so \(P'\) is
decreasing.  Direct substitution of \eqref{eq:u-sigma} into
\eqref{eq:P-quadratic} gives \eqref{eq:P-prime-boundary}.  Hence
\(P'(u)\ge P'(u_\sigma)>0\) throughout the cap.
\end{proof}

\begin{lemma}[First-lobe geometry]\label{lem:first-lobe-geometry}
Let
\[
0<\zeta_1<\cdots<\zeta_n<B
\]
be the zeros of \(H\) in the \(u\)-coordinate.  On the first lobe,
\(0<u<\zeta_1\), one has
\begin{equation}\label{eq:first-lobe-factorization}
H(u)=K\,u^{a/2}(B-u)^{b/2}
\prod_{j=1}^n(\zeta_j-u),
\qquad K>0.
\end{equation}
Moreover,
\begin{equation}\label{eq:first-lobe-log-concavity}
(\log H)''(u)
=-\frac{a}{2u^2}-\frac{b}{2(B-u)^2}
-\sum_{j=1}^n\frac1{(\zeta_j-u)^2}<0.
\end{equation}
Consequently, \(H\) has a unique critical point
\(u_1\in(0,\zeta_1)\); it is the positive maximum of the first lobe.
\end{lemma}

\begin{proof}
The standard zero theorem for Jacobi polynomials places all transformed
zeros strictly between \(0\) and \(B\).  Combining the resulting
polynomial factorization with the two weight factors in \(H\) gives
\eqref{eq:first-lobe-factorization}; the sign is positive because
\(P_n^{(a,b)}(1)>0\).  Two differentiations give
\eqref{eq:first-lobe-log-concavity}.  Thus \((\log H)'\) is strictly
decreasing, tends to \(+\infty\) at \(0\), and tends to \(-\infty\)
at \(\zeta_1\), proving the final assertion.
\end{proof}

\begin{lemma}[Sonin identity]\label{lem:sonin}
On any interval on which \(Q>0\), define
\[
\mathcal S(u)=H(u)^2+\frac{p(u)H'(u)^2}{Q(u)}.
\]
Then
\begin{equation}\label{eq:sonin-derivative}
\mathcal S'(u)
=-\frac{P'(u)}{Q(u)^2}H'(u)^2.
\end{equation}
In particular, \(\mathcal S\) is nonincreasing on every subinterval of
the endpoint cap on which \(Q>0\).
\end{lemma}

\begin{proof}
Differentiate \(\mathcal S\), use
\(pH''+p'H'=-QH\), and recall that \(P=pQ\).  The mixed terms cancel,
leaving
\[
\mathcal S'
=-\frac{p'Q+pQ'}{Q^2}H'^2
=-\frac{P'}{Q^2}H'^2.
\]
\end{proof}

\begin{proposition}[First-lobe reduction]\label{prop:first-lobe-reduction}
Let \(\zeta_1\) be the first zero of \(H\) in the \(u\)-coordinate,
and let \(u_1\in(0,\zeta_1)\) be the maximum point of the first lobe,
with \(A=H(u_1)>0\).

\begin{enumerate}[label=\textup{(\roman*)}]
\item If \(P(u_\sigma)\le0\), then \(H\) is nonnegative on
\([0,u_\sigma]\), positive on \((0,u_\sigma]\), and strictly increasing;
hence the right cap is controlled by
\Cref{lem:central}.
\item If \(P(u_\sigma)>0\), then either \(u_1>u_\sigma\), in which case
the cap is again increasing up to its central boundary, or
\(u_1\le u_\sigma\), in which case
\begin{equation}\label{eq:sonin-pointwise}
\lvert H(u)\rvert\le A,\qquad 0\le u\le u_\sigma.
\end{equation}
\end{enumerate}
\end{proposition}

\begin{proof}
Since \(P(0)=-a^2/4<0\) and \(p>0\) on the open cap, \(P\) and \(Q\)
have the same sign there.  Near \(u=0\),
\[
H(u)\sim C u^{a/2}>0,\qquad
w(u):=p(u)H'(u)\longrightarrow0,
\]
with \(w(u)>0\) for small \(u>0\).  Moreover,
\(Q(u)H(u)=O(u^{a/2-1})\), which is integrable at zero.

If \(P(u_\sigma)\le0\), the strict increase of \(P\) implies \(Q\le0\)
throughout the open cap.  As long as \(H>0\),
\[
w'=-QH\ge0.
\]
Thus \(w\ge0\) and \(H'=w/p\ge0\).  A first-zero continuation argument
shows that neither \(H\) nor \(w\) can lose its sign.  This proves (i).

Suppose \(P(u_\sigma)>0\).  By \Cref{lem:P-monotone}, \(P\) has exactly
one zero \(u_0\in(0,u_\sigma)\).  The same argument shows that \(H\) is
positive and increasing through \(u_0\).  By
\Cref{lem:first-lobe-geometry}, the unique first-lobe critical point
is \(u_1\), and therefore \(u_1\in(u_0,\zeta_1)\).  If
\(u_1>u_\sigma\), there is no critical point
inside the cap and the cap is increasing.  If \(u_1\le u_\sigma\), then
\(Q>0\) on \([u_1,u_\sigma]\), so \Cref{lem:sonin} gives
\[
H(u)^2\le\mathcal S(u)\le\mathcal S(u_1)=A^2,
\qquad u_1\le u\le u_\sigma.
\]
On \([0,u_1]\), the definition of the first critical point gives
\(0\le H(u)\le A\).  This proves \eqref{eq:sonin-pointwise}.
\end{proof}

\section{The remaining endpoint regime}\label{sec:overflow}

The Sonin reduction leaves only the overflow case
\(P(u_\sigma)>0\).  The next lemma confines it to the parameter range
in which the first-lobe estimate can be compared with the KKT target.

\begin{lemma}[Overflow restriction]\label{lem:overflow}
If \(P(u_\sigma)>0\), then
\begin{equation}\label{eq:overflow-restriction}
a<n+\sqrt{4n^2+2n}<3(n+1).
\end{equation}
\end{lemma}

\begin{proof}
Substitution of \eqref{eq:u-sigma} into \eqref{eq:P-quadratic} gives
\begin{equation}\label{eq:P-at-usigma}
P(u_\sigma)=\frac{\Phi_R(n,a,b)}{4(a+b+2n)},
\end{equation}
where
\begin{equation}\label{eq:Phi-R}
\begin{split}
\Phi_R(n,a,b)={}&
b(-a^2+2na+3n^2+2n)\\
&-a^3+(3n^2+2n)a+2n^3+2n^2.
\end{split}
\end{equation}
Set
\[
a_\ast=n+\sqrt{4n^2+2n}.
\]
The coefficient of \(b\) in \eqref{eq:Phi-R} is nonpositive for
\(a\ge a_\ast\).  At \(a=a_\ast\), the remaining term equals
\[
-2n^2(2a_\ast+2n+1)<0,
\]
and its derivative with respect to \(a\) is
\[
-3a^2+3n^2+2n<0,\qquad a\ge a_\ast.
\]
Thus \(\Phi_R(n,a,b)<0\) for \(a\ge a_\ast\) and \(b\ge0\).  Hence
overflow forces \(a<a_\ast\).  Finally,
\[
\sqrt{4n^2+2n}<2n+3
\]
for \(n\ge1\), which gives \(a_\ast<3(n+1)\).
\end{proof}

\section{The inverse moment and a canonical-product principle}
\label{sec:first-lobe}

The first-lobe maximum is controlled by combining the exact Jacobi
inverse moment with a canonical-product argument that does not use the
differential equation.

\subsection{The inverse moment}

\begin{lemma}[Exact inverse moment]\label{lem:inverse-moment}
For \(n\ge1\), \(a>0\), and \(b\ge0\), the function \(H\) in
\eqref{eq:B-H-p} satisfies
\begin{equation}\label{eq:inverse-moment}
\int_0^B\frac{H(u)^2}{u}\,du=\frac1a.
\end{equation}
\end{lemma}

\begin{proof}
Put
\[
\mathcal P(t)=P_n^{(a,b)}(1-2t).
\]
Rodrigues' formula in the \(t\)-coordinate is
\begin{equation}\label{eq:rodrigues-t}
\mathcal P(t)=\frac1{n!}t^{-a}(1-t)^{-b}
\frac{d^n}{dt^n}\left[t^{n+a}(1-t)^{n+b}\right].
\end{equation}
Let
\[
I=\int_0^1t^{a-1}(1-t)^b\mathcal P(t)^2\,dt,
\qquad
U(t)=\frac{\mathcal P(t)}t,\qquad
F(t)=t^{n+a}(1-t)^{n+b}.
\]
Insert \eqref{eq:rodrigues-t} for one copy of \(\mathcal P\), first
integrating on \([\varepsilon,1-\varepsilon]\).  For
\(0\le j\le n-1\), the boundary products satisfy
\begin{align}
U^{(j)}(t)F^{(n-1-j)}(t)&=O(t^a),
&&t\downarrow0, \label{eq:ibp-zero}\\
U^{(j)}(t)F^{(n-1-j)}(t)
&=O((1-t)^{b+1+j}),
&&t\uparrow1. \label{eq:ibp-one}
\end{align}
Thus all boundary terms vanish as \(\varepsilon\downarrow0\), including
when \(b=0\).  Since
\[
\frac{\mathcal P(t)}t
=\frac{\mathcal P(0)}t+Q_{n-1}(t),
\]
where \(Q_{n-1}\) is a polynomial of degree at most \(n-1\), we have
\[
U^{(n)}(t)=(-1)^n n!\,\mathcal P(0)t^{-n-1}.
\]
Consequently,
\begin{align*}
I
&=\frac{(-1)^n}{n!}\int_0^1U^{(n)}(t)F(t)\,dt\\
&=\mathcal P(0)B(a,n+b+1)\\
&=\frac{\G(n+a+1)\G(n+b+1)}
{a\,n!\,\G(n+a+b+1)}.
\end{align*}
Multiplying by the normalization in \eqref{eq:def-g} and changing
variables \(u=Bt\) give \eqref{eq:inverse-moment}.
\end{proof}

\subsection{The canonical-product principle}

The next theorem is independent of a differential equation.  It
isolates the real-zero product, the endpoint factor, and the singular
inverse moment; the countable case will also give a Bessel application.

\begin{theorem}[Canonical-product first-lobe principle]
\label{thm:canonical-product-lobe}
Let \(a>0\), \(\lambda\ge0\), and \(C>0\).  Consider one of the
following cases.

\smallskip
\noindent
\emph{Finite endpoint.}  There are \(B>0\), \(b\ge0\), \(N\ge1\),
and \(0<\zeta_1<\cdots<\zeta_N<B\) such that
\begin{equation}\label{eq:canonical-finite}
\Phi(u)=C u^{a/2}(B-u)^{b/2}\e^{-\lambda u/2}
\prod_{j=1}^{N}\left(1-\frac{u}{\zeta_j}\right),
\qquad 0<u<B.
\end{equation}

\smallskip
\noindent
\emph{Infinite endpoint.}  On \((0,\infty)\),
\begin{equation}\label{eq:canonical-infinite}
\Phi(u)=C u^{a/2}\e^{-\lambda u/2}
\prod_{j\in\mathcal J}\left(1-\frac{u}{\zeta_j}\right),
\end{equation}
where \(\mathcal J\) is finite or countable.  For a nonempty product,
\(0<\zeta_1<\zeta_2<\cdots\); in the countable case,
\(\sum_j\zeta_j^{-1}<\infty\).  The empty product is allowed if
\(\lambda>0\).

If the product is empty put \(\zeta_1=\infty\), and in either case
understand \(\Phi(0)=0\) by continuous extension.  Then \(\Phi\) has a
unique maximum \(A=\Phi(u_1)\) on \((0,\zeta_1)\), and
\begin{equation}\label{eq:canonical-envelope}
\frac{\Phi(tu_1)}{A}
\ge t^{a/2}\left(1+\frac a2(1-t)\right),
\qquad 0\le t\le1+\frac2a.
\end{equation}
Consequently,
\begin{equation}\label{eq:first-lobe-mass}
\int_0^{u_1}\frac{\Phi(u)^2}{u}\,du
\ge\frac{A^2}{a}R(a),
\qquad
R(a)=1+\frac{a}{a+1}
+\frac{a^2}{2(a+1)(a+2)}.
\end{equation}
Equivalently,
\begin{equation}\label{eq:R-def}
R(a)=\frac{5a^2+10a+4}{2(a+1)(a+2)}.
\end{equation}
Define also
\begin{equation}\label{eq:Rstar-def}
R_*(a)
:=a\int_0^{1+2/a}t^{a-1}
\left(1+\frac a2(1-t)\right)^2\,dt
=\frac{a+2}{2(a+1)}\left(\frac{a+2}{a}\right)^a.
\end{equation}
Then \(R_*(a)>R(a)\).  Put \(L=B\) in the finite-endpoint case and
\(L=\infty\) in the infinite-endpoint case.  If the full inverse moment
\[
I:=\int_0^L\frac{\Phi(u)^2}{u}\,du
\]
is finite, then
\begin{equation}\label{eq:canonical-peak-bound}
A^2<\frac{aI}{R_*(a)}.
\end{equation}
\end{theorem}

\begin{proposition}[Sharpness of the canonical-product constants]
\label{prop:canonical-sharpness}
Under the hypotheses of \Cref{thm:canonical-product-lobe},
the constant in \eqref{eq:canonical-envelope} is pointwise best
possible, and \(R(a)\) is the largest universal constant in
\eqref{eq:first-lobe-mass}.  Equality in \eqref{eq:first-lobe-mass}
holds precisely when \(\lambda=0\), the product has one factor, and,
in the finite-endpoint case, \(b=0\).  The constant \(R_*(a)\) in
\eqref{eq:canonical-peak-bound} is also best possible over the stated
class, although equality in \eqref{eq:canonical-peak-bound} is never
attained.  Here optimality is for the union of the finite- and
infinite-endpoint cases; the extremizing sequence already lies in the
finite-endpoint degree-one family.
No optimality assertion is made here for the infinite-endpoint
subclass considered by itself.
\end{proposition}

\begin{proof}[Proof of \Cref{thm:canonical-product-lobe}]
In the countable case, on every compact subinterval of
\([0,\zeta_1)\), the product and the first two logarithmic derivatives
converge uniformly.  This follows from \(\sum_j\zeta_j^{-1}<\infty\)
and the bound
\[
\frac1{\zeta_j-u}
\le\frac1{1-\rho/\zeta_1}\frac1{\zeta_j},
\qquad 0\le u\le\rho<\zeta_1.
\]
Moreover,
\[
\left|\log\left(1-\frac u{\zeta_j}\right)\right|
\le \frac{\rho}{1-\rho/\zeta_1}\frac1{\zeta_j},
\qquad
\frac1{(\zeta_j-u)^2}
\le\frac1{(1-\rho/\zeta_1)^2}\frac1{\zeta_j^2}.
\]
Since \(\sum_j\zeta_j^{-2}<\infty\) as well, all following
differentiations are valid also for a countable product.

On the first lobe,
\[
(\log\Phi)'(u)=\frac{a}{2u}-\frac{b}{2(B-u)}
-\frac\lambda2-\sum_j\frac1{\zeta_j-u},
\]
with the \(b\)-term omitted at the infinite endpoint, and
\[
(\log\Phi)''(u)
=-\frac{a}{2u^2}-\frac{b}{2(B-u)^2}
-\sum_j\frac1{(\zeta_j-u)^2}<0.
\]
The logarithmic derivative tends to \(+\infty\) at zero and to
\(-\infty\) at the first zero; for an empty product it tends to
\(-\lambda/2<0\) at infinity.  Hence \(u_1\) exists and is unique.

Put
\[
x=1-t,\quad d=\frac{\lambda u_1}{2},\quad
r_j=\frac{u_1}{\zeta_j-u_1},\quad c=\sum_jr_j.
\]
At a finite endpoint also put
\(p=b/2\) and \(q=u_1/(B-u_1)\); otherwise put \(p=q=0\).
The critical-point identity is
\begin{equation}\label{eq:canonical-criticality}
\frac a2=pq+d+c.
\end{equation}
The sum defining \(c\) converges in the countable case, and
\begin{align}
\frac{\Phi(tu_1)}A
&=t^{a/2}(1+qx)^p\e^{dx}\prod_j(1+r_jx)\notag\\
&\ge t^{a/2}(1+qx)^p\e^{dx}(1+cx).
\label{eq:canonical-ratio}
\end{align}
If \(p>0\), every \(\zeta_j<B\) gives \(r_j>q\), hence \(c>q\).
For \(f(x)=(1+qx)^p(1+cx)\), direct differentiation gives
\[
f''(x)=pq(1+qx)^{p-2}
\left(2c+(p-1)q+cq(p+1)x\right)\ge0.
\]
For \(0<p<1\), the only non-immediate case, the constant factor is
larger than \(2q-(1-p)q=(1+p)q\).  Thus \(f,f',f''\ge0\), and
\(h(x)=\e^{dx}f(x)\) is convex.  By
\eqref{eq:canonical-criticality},
\[
h(0)=1,\qquad h'(0)=d+pq+c=\frac a2.
\]
Its tangent inequality in \eqref{eq:canonical-ratio} proves
\eqref{eq:canonical-envelope} for \(0\le t\le1\).

It remains to continue the same envelope down the decreasing side of
the first lobe.  Put \(S=a/2\), write \(t=1+y\), and suppose first that
the product is nonempty.  Since
\[
r_1\le c\le S,
\]
the interval \(0\le y\le S^{-1}\) lies before the first zero; at a
finite endpoint it also lies before \(B\), because \(q<r_1\).  All
factors below are therefore nonnegative, and finite partial products,
followed by passage to the limit when necessary, give
\begin{align}
\frac{\Phi((1+y)u_1)}{A(1+y)^S}
&=(1-qy)^p\e^{-dy}\prod_j(1-r_jy)\notag\\
&\ge \e^{-dy}(1-qy)^p(1-cy).
\label{eq:canonical-right-ratio}
\end{align}
Set \(f(y)=(1-qy)^p(1-cy)\).  If \(p=0\), this function is linear.
If \(p>0\), write \(c=qz\), where \(z>1\).  A calculation gives
\[
f''(y)=pq(1-qy)^{p-2}
\left(2c+(p-1)q-cq(p+1)y\right).
\]
Since \(S\ge q(p+z)\) and \(y\le S^{-1}\), the expression in
parentheses is bounded below by
\[
\frac{q}{p+z}
\left(p^2-p+2pz+2z^2-2z\right)>0.
\]
Indeed, the polynomial in parentheses equals \(p^2+p\) at \(z=1\)
and is strictly increasing for \(z\ge1\).  Thus in all cases
\(f\ge0\), \(f'\le0\), and \(f''\ge0\).  Hence
\[
h(y)=\e^{-dy}f(y),\qquad
h''(y)=\e^{-dy}\bigl(f''-2df'+d^2f\bigr)\ge0.
\]
By \eqref{eq:canonical-criticality}, \(h(0)=1\) and
\(h'(0)=-S\).  The tangent inequality for \(h\), inserted in
\eqref{eq:canonical-right-ratio}, proves
\eqref{eq:canonical-envelope} for \(1\le t\le1+2/a\).
For an empty product, criticality gives \(d=S\), and the same conclusion
is simply \(\e^{-Sy}\ge1-Sy\).

Squaring and putting \(u=tu_1\) first gives
\begin{align*}
\int_0^{u_1}\frac{\Phi(u)^2}{u}\,du
&\ge A^2\int_0^1t^{a-1}
\left(1+\frac a2(1-t)\right)^2dt
=\frac{A^2}{a}R(a).
\end{align*}
Integration over the full range of \eqref{eq:canonical-envelope} gives
\[
I\ge \frac{A^2}{a}R_*(a).
\]
The endpoint of the domain lies strictly beyond
\((1+2/a)u_1\), and \(\Phi^2\) has positive mass on a set of positive
measure beyond that point.  This proves the strict inequality
\eqref{eq:canonical-peak-bound}.  Formula \eqref{eq:Rstar-def} follows
by elementary integration, and
\[
R_*(a)-R(a)
=a\int_1^{1+2/a}t^{a-1}
\left(1+\frac a2(1-t)\right)^2\,dt>0.
\]
\end{proof}

\begin{proof}[Proof of \Cref{prop:canonical-sharpness}]
Use the notation introduced in the preceding proof.  If the product has at least
two factors, then for every finite partial product containing its first
two factors,
\[
\prod_{j=1}^{N}(1+r_jx)
\ge 1+x\sum_{j=1}^{N}r_j+x^2r_1r_2.
\]
Passing to the limit in the countable case leaves a positive quadratic
cross term for \(x>0\).  The exponential tangent comparison is strict
when \(d>0\), and the finite-endpoint tangent comparison is strict when
\(p>0\).  It follows that equality in \eqref{eq:first-lobe-mass} has
exactly the stated one-factor form.  That form makes
\eqref{eq:canonical-envelope} an identity and proves pointwise
sharpness as well.

Finally, for \(\varepsilon>0\), take the finite-endpoint example
\[
\Phi_\varepsilon(u)=C_\varepsilon u^{a/2}(1-u),
\qquad 0<u<1+\varepsilon,
\]
with \(b=\lambda=0\) and sole zero \(\zeta_1=1\).  Its maximum occurs
at \(u_1=a/(a+2)\), and its normalized profile agrees with the
right-hand side of \eqref{eq:canonical-envelope} up to its first zero.
Consequently,
\[
\frac{a}{A_\varepsilon^2}
\int_0^{1+\varepsilon}\frac{\Phi_\varepsilon(u)^2}{u}\,du
\longrightarrow R_*(a)
\qquad(\varepsilon\downarrow0).
\]
Thus \(R_*(a)\) cannot be increased, although the preceding strictness
argument shows that it is not attained in the stated class.
\end{proof}

\subsection{Jacobi consequences}

\begin{corollary}[First Jacobi peak]\label{cor:first-peak}
Let \(A\) be the first-lobe maximum of \(H\).  Then
\begin{equation}\label{eq:first-peak}
A^2<\frac1{R_*(a)}.
\end{equation}
\end{corollary}

\begin{proof}
If \(0<\zeta_1<\cdots<\zeta_n<B\) are the transformed Jacobi zeros,
then before \(\zeta_1\)
\[
H(u)=C\,u^{a/2}(B-u)^{b/2}
\prod_{j=1}^n(\zeta_j-u)
\]
with \(C>0\).  Apply \Cref{thm:canonical-product-lobe} and bound its left-hand
side above by the full integral in \Cref{lem:inverse-moment}.
\end{proof}

\begin{proposition}[Two-sided lobe monotonicity]
\label{prop:two-sided-lobes}
Assume \(n\ge1\) and \(a,b>0\).  Put \(\zeta_0=0\) and
\(\zeta_{n+1}=B\), and let \(u^{(j)}\) be the unique maximum point of
\(|H|\) on \((\zeta_j,\zeta_{j+1})\), with
\(\mu_j=|H(u^{(j)})|\), \(0\le j\le n\).  If \(u_v\) is the vertex
of the concave quadratic \(P\) in \eqref{eq:P-quadratic}, then
\begin{align}
\mu_j&\ge\mu_{j+1}
&&\text{whenever }u^{(j+1)}\le u_v,\label{eq:left-lobe-monotonicity}\\
\mu_j&\le\mu_{j+1}
&&\text{whenever }u^{(j)}\ge u_v.\label{eq:right-lobe-monotonicity}
\end{align}
In particular,
\begin{equation}\label{eq:extreme-lobe-identity}
\max_{0\le u\le B}|H(u)|=\max\{\mu_0,\mu_n\}.
\end{equation}
\end{proposition}

\begin{proof}
On every nodal interval, the real-zero factorization gives
\[
(\log|H|)''(u)
=-\frac{a}{2u^2}-\frac{b}{2(B-u)^2}
-\sum_{k=1}^n\frac1{(u-\zeta_k)^2}<0.
\]
Thus each \(u^{(j)}\) exists and is unique.  Moreover, at such a
point \(H'=0\) and \(HH''<0\).  Since
\(pH''=-QH\) there, every lobe maximum lies in the open set
\(\{Q>0\}=\{P>0\}\).

For clarity, this set is nonempty.  Indeed,
\begin{align*}
P(0)&=-\frac{a^2}{4},& P(B)&=-\frac{b^2}{4},\\
P'(0)&=\nu+\frac{ra}{2}>0,&
P'(B)&=-\nu-\frac{rb}{2}<0.
\end{align*}
Hence \(u_v\in(0,B)\).  If \(P(u_v)\le0\), then \(Q\le0\) on
\((0,B)\).  Starting from \(H(u)\sim C u^{a/2}>0\) and
\(pH'\to0\) at zero, the first-zero argument used in
\Cref{prop:first-lobe-reduction} would make \(H\) positive and
increasing up to its first zero, a contradiction.  Therefore
\(P(u_v)>0\), and \(\{P>0\}=(u_-,u_+)\) for two simple roots
\(0<u_-<u_v<u_+<B\).

The Sonin function of \Cref{lem:sonin} is consequently defined at and
between all lobe maxima.  It is nonincreasing on \((u_-,u_v)\) and
nondecreasing on \((u_v,u_+)\), because \(P'\) is positive to the
left of the vertex and negative to the right.  At a lobe maximum,
\(\mathcal S(u^{(j)})=\mu_j^2\).  Comparing consecutive maxima on
each side of \(u_v\) proves
\eqref{eq:left-lobe-monotonicity}--\eqref{eq:right-lobe-monotonicity}.
Every maximum on the left is bounded by \(\mu_0\), and every maximum
on the right by \(\mu_n\), which proves
\eqref{eq:extreme-lobe-identity}.
\end{proof}

\section{Completion for positive degree}\label{sec:completion}

It remains to compare the universal first-lobe bound with the
parameter-dependent KKT factor and then assemble the endpoint caps with
the central interval.

\begin{lemma}[Comparison with the target]\label{lem:target-comparison}
Suppose \(P(u_\sigma)>0\).  Then the first-lobe maximum satisfies
\[
A<T_{n,a,b}.
\]
\end{lemma}

\begin{proof}
Put \(m=n+1\) and
\[
z=\frac{ab}{m(m+a+b)}.
\]
Then
\begin{equation}\label{eq:T-z}
T_{n,a,b}^4
=\frac{m(m+a+b)}{(m+a)(m+b)}
=\frac1{1+z}.
\end{equation}
By \Cref{lem:overflow}, \(a<3m\), while \(m\ge2\).  Hence
\begin{equation}\label{eq:z-bounds}
z<\frac am\le\frac a2,\qquad z<3.
\end{equation}
For \(0<a\le6\), direct factorization gives
\begin{equation}\label{eq:R-small-a}
R(a)^2-\left(1+\frac a2\right)
=-\frac{a(2a+3)(a^3-6a^2-16a-8)}
{4(a+1)^2(a+2)^2}>0,
\end{equation}
because
\[
a^3-6a^2-16a-8=a^2(a-6)-16a-8<0.
\]
Thus \(R(a)^2>1+a/2>1+z\).  For \(a\ge6\),
\begin{equation}\label{eq:R-large-a}
R(a)-2=\frac{a^2-2a-4}{2(a+1)(a+2)}>0,
\end{equation}
and therefore \(R(a)^2>4>1+z\).  In both cases, \(R_*(a)>R(a)\) and
\[
A^4<\frac1{R_*(a)^2}
<\frac1{R(a)^2}
<\frac1{1+z}=T_{n,a,b}^4
\]
by \Cref{cor:first-peak,eq:T-z}.  Since \(A\ge0\), the result follows.
\end{proof}

\begin{proof}[Proof of \Cref{thm:main} for \(n\ge1\)]
Assume first that \(\alpha,\beta>0\).  By \Cref{lem:central}, the target
holds on the central interval.  On the right cap, set
\(a=\alpha\) and \(b=\beta\).  If \(P(u_\sigma)\le0\),
\Cref{prop:first-lobe-reduction}(i) reduces the cap to its central
boundary.  If \(P(u_\sigma)>0\), part (ii) of that proposition reduces
the cap either to the same boundary or to the first-lobe maximum, which
is controlled by \Cref{lem:target-comparison}.  Thus the right cap is
settled.  The reflection identity \eqref{eq:reflection}, with
\(\alpha\) and \(\beta\) interchanged, gives the left cap.

Suppose now that one or both parameters vanish.  For every
\(-1<x<1\), approximate each zero parameter from above, fixing the
other parameter; at the corner use a simultaneous positive-parameter
limit.  The gamma factors, Jacobi coefficients, and endpoint weights
are continuous in the parameters in the open interval.  At the
endpoints, compute directly:
\[
g_n^{(0,\beta)}(1)=1,\qquad
g_n^{(\alpha,0)}(-1)=(-1)^n,
\]
and \(T_{n,\alpha,\beta}=1\) on either parameter axis.  The opposite
endpoints vanish when the corresponding exponent is positive.  At
\((\alpha,\beta)=(0,0)\), the endpoint values are \(1\) and
\((-1)^n\).  This proves the theorem for \(n\ge1\).
\end{proof}

\begin{proof}[Proof of \Cref{thm:intro-extreme-lobes}]
Assume first that \(\alpha,\beta>0\).  In the coordinate
\(u=B(1-x)/2\), \Cref{prop:two-sided-lobes} gives
\[
\max_{-1\le x\le1}|g_n^{(\alpha,\beta)}(x)|
=\max\{\mu_0,\mu_n\}.
\]
The first-lobe estimate \Cref{cor:first-peak} gives
\(\mu_0<R_*(\alpha)^{-1/2}\).  Applying the same estimate after the
reflection \eqref{eq:reflection} gives
\(\mu_n<R_*(\beta)^{-1/2}\).  Since
\(R_*(a)>R(a)>1\) for \(a>0\), this proves
\eqref{eq:intro-extreme-lobes}.

If a parameter is zero, pass to the limit from positive parameters at
each \(-1<x<1\) and evaluate the endpoints directly, exactly as in the
preceding proof.  Since \(R_*(a)\to1\) as \(a\downarrow0\), this gives
the asserted non-strict bound on the closed quadrant.  Finally,
intersecting that bound with \Cref{thm:main} proves
\eqref{eq:intro-combined-bound}.
\end{proof}

\section{The degree-zero case}\label{sec:degree-zero}

The preceding argument assumed positive degree.  When \(n=0\), the
claim instead reduces to a gamma-function inequality.

\begin{lemma}[A gamma-function inequality]\label{lem:gamma}
For \(r\ge0\), define
\[
h(r)=\log\G(r+1)-r\log r-\frac12\log(r+1),
\qquad h(0)=0.
\]
Then \(h\) is concave and, for \(\alpha,\beta\ge0\),
\begin{equation}\label{eq:h-subadditive}
h(\alpha+\beta)\le h(\alpha)+h(\beta).
\end{equation}
\end{lemma}

\begin{proof}
For \(r>0\), midpoint convexity for \(y\mapsto(r+y)^{-2}\) gives
\[
\frac1{(r+k)^2}
\le\int_{k-1/2}^{k+1/2}\frac{dy}{(r+y)^2},
\qquad k\ge1,
\]
and hence
\[
\psi_1(r+1)=\sum_{k=1}^{\infty}\frac1{(r+k)^2}
\le\int_{1/2}^{\infty}\frac{dy}{(r+y)^2}
=\frac1{r+1/2}.
\]
Since
\[
\frac1r-\frac1{2(r+1)^2}-\frac1{r+1/2}
=\frac{3r+2}{2r(2r+1)(r+1)^2}>0,
\]
we obtain
\[
h''(r)=\psi_1(r+1)-\frac1r+\frac1{2(r+1)^2}<0.
\]
Moreover, \(h\) extends continuously to \(h(0)=0\).  Concavity applied
to the two chords from zero to \(\alpha+\beta\) yields
\[
h(\alpha)\ge
\frac{\alpha}{\alpha+\beta}h(\alpha+\beta),
\qquad
h(\beta)\ge
\frac{\beta}{\alpha+\beta}h(\alpha+\beta).
\]
Adding proves \eqref{eq:h-subadditive}; the boundary cases follow by
continuity.
\end{proof}

\begin{proof}[Proof of \Cref{thm:main} for \(n=0\)]
Put \(t=(1-x)/2\) and \(s=\alpha+\beta\).  Since \(P_0=1\),
\begin{equation}\label{eq:g-zero}
g_0^{(\alpha,\beta)}(x)^2
=\frac{\G(s+1)}{\G(\alpha+1)\G(\beta+1)}
t^\alpha(1-t)^\beta.
\end{equation}
For positive \(\alpha,\beta\), the maximum of the last weight is
\[
\frac{\alpha^\alpha\beta^\beta}{s^s}.
\]
Rearranging \eqref{eq:h-subadditive} gives
\[
\frac{\G(s+1)}{\G(\alpha+1)\G(\beta+1)}
\le
\frac{s^s}{\alpha^\alpha\beta^\beta}
\sqrt{\frac{s+1}{(\alpha+1)(\beta+1)}}.
\]
Consequently,
\[
\max_{-1\le x\le1}g_0^{(\alpha,\beta)}(x)^2
\le
\sqrt{\frac{s+1}{(\alpha+1)(\beta+1)}}
=T_{0,\alpha,\beta}^2.
\]
The parameter-axis cases follow directly or by continuity.  This
completes the proof of \Cref{thm:main}.
\end{proof}

The unrefined contractivity statement also settles the second
Bernstein-type inequality that Koornwinder, Kostenko and Teschl made
conditional on their conjecture.

\begin{corollary}[One-sided Jacobi and Meixner bounds]
\label{cor:one-sided-jacobi-meixner}
For \(n\in\N_0\), \(\alpha,\beta\ge0\), and \(-1\le x\le1\),
\begin{equation}\label{eq:one-sided-jacobi}
\left(\frac{1+x}{2}\right)^{\beta/2}
|P_n^{(\alpha,\beta)}(x)|
\le \binom{n+\alpha}{n}.
\end{equation}
Equivalently, if
\[
M_n(y;\gamma,c)
={}_2F_1\!\left(\begin{matrix}-n,-y\\ \gamma\end{matrix};
1-c^{-1}\right),
\]
then
\begin{equation}\label{eq:meixner-one-sided}
c^{(n+y)/2}|M_n(y;\gamma,c)|\le1
\end{equation}
for \(n\in\N_0\), real \(y\ge n\), \(\gamma\ge1\), and
\(0<c<1\).
\end{corollary}

\begin{proof}
Theorem~5.6 of \cite{KKT2018} states that
\eqref{eq:one-sided-jacobi} holds for every pair of parameters for
which the contractivity inequality \eqref{eq:basic-intro} holds.
Apply \Cref{thm:main}.  The Meixner formulation is the equivalent form
given in \cite[Equations~(3.5) and~(5.2)]{KKT2018}; the full Jacobi
quadrant corresponds to \(\gamma\ge1\).
\end{proof}

\section{A one-sided inequality for negative parameters}
\label{sec:negative-jacobi}

With \Cref{thm:main} complete, we turn to the negative-parameter regime
needed for the second half of the dispersive phase law.  When the first
Jacobi parameter is negative, the right-edge factor in
\eqref{eq:def-g} is singular.  The appropriate contractive quantity
therefore omits that factor.  In the discrete Laguerre
evolution, this is the factor responsible for the slower decay
predicted in \cite[Remark~6.6]{KKT2018}.

For \(a>-1\), \(d\ge0\), and \(n\in\N_0\), put
\begin{equation}\label{eq:def-B-negative}
\mathcal B_{n,d}^{(a)}(x)=
\left[
\frac{\G(n+1)\G(n+a+d+1)}
{\G(n+a+1)\G(n+d+1)}
\right]^{1/2}
\left(\frac{1+x}{2}\right)^{d/2}
P_n^{(a,d)}(x).
\end{equation}

\begin{theorem}[Negative-parameter contractivity]
\label{thm:negative-contractivity}
If \(-1<a\le0\) and \(d\ge0\), then, for every \(n\in\N_0\) and
\(x\in[-1,1]\),
\begin{equation}\label{eq:negative-contractivity}
\left|\mathcal B_{n,d}^{(a)}(x)\right|\le1.
\end{equation}
The constant is sharp: when \(d=0\), equality holds at \(x=-1\).
\end{theorem}

The parameter \(d\) is allowed to be real in
\Cref{thm:negative-contractivity}; only \(d\in\N_0\) is needed for
the spectral application.  We first record the reflected inverse
moment that drives the proof.

\begin{lemma}[Reflected inverse moment]
\label{lem:reflected-inverse-moment}
Let \(a>-1\), \(d>0\), and \(n\in\N_0\).  If
\[
Y(u)=\mathcal B_{n,d}^{(a)}(2u-1),\qquad 0\le u\le1,
\]
then
\begin{equation}\label{eq:reflected-inverse-moment}
\int_0^1(1-u)^a\frac{Y(u)^2}{u}\,du=\frac1d.
\end{equation}
\end{lemma}

\begin{proof}
By reflection,
\[
P_n^{(a,d)}(2u-1)=(-1)^nP_n^{(d,a)}(1-2u).
\]
Put \(\mathcal P(u)=P_n^{(d,a)}(1-2u)\).  The Rodrigues
integration used in \Cref{lem:inverse-moment} applies with first
parameter \(d>0\) and second parameter \(a>-1\).  At \(u=1\), its
boundary products are \(O((1-u)^{a+1+j})\), \(0\le j<n\), and hence
still vanish; at \(u=0\), they are \(O(u^d)\).  It gives
\begin{equation}\label{eq:reflected-polynomial-moment}
\int_0^1u^{d-1}(1-u)^a\mathcal P(u)^2\,du
=
\frac{\G(n+d+1)\G(n+a+1)}
{d\,n!\,\G(n+a+d+1)}.
\end{equation}
Multiplication by the square of the normalization in
\eqref{eq:def-B-negative} proves
\eqref{eq:reflected-inverse-moment}.  For \(n=0\), the formula also
follows directly from the beta integral.
\end{proof}

\begin{lemma}[First negative-parameter peak]
\label{lem:negative-first-peak}
Assume \(-1<a\le0\), \(d>0\), and \(n\ge1\).  Let \(A\) be the
maximum of \(|\mathcal B_{n,d}^{(a)}|\) on the lobe adjacent to
\(x=-1\).  Then
\begin{equation}\label{eq:negative-first-peak}
A^2\le \frac1{R(d)}<1.
\end{equation}
\end{lemma}

\begin{proof}
Let \(0<\zeta_1<\cdots<\zeta_n<1\) be the Jacobi zeros in the
coordinate \(u=(1+x)/2\).  After changing the overall sign, on the
first lobe
\[
Y(u)=\mathcal B_{n,d}^{(a)}(2u-1)
=C u^{d/2}\prod_{j=1}^n(\zeta_j-u)>0.
\]
Apply \Cref{thm:canonical-product-lobe} with endpoint exponent \(d\), with
the other endpoint exponent equal to zero, and with \(B=1\).  If
\(u_1\) is the unique maximum point, then
\[
\int_0^{u_1}\frac{Y(u)^2}{u}\,du
\ge \frac{A^2}{d}R(d).
\]
On the other hand, \((1-u)^a\ge1\) because \(a\le0\), so
\Cref{lem:reflected-inverse-moment} yields
\[
\int_0^{u_1}\frac{Y(u)^2}{u}\,du
\le \int_0^1(1-u)^a\frac{Y(u)^2}{u}\,du
=\frac1d.
\]
Comparison proves \eqref{eq:negative-first-peak}.
\end{proof}

\begin{proof}[Proof of \Cref{thm:negative-contractivity}]
Write \(Y(x)=\mathcal B_{n,d}^{(a)}(x)\).  For \(s\ge0\), set
\[
c_a(s)=\frac{\G(s+a+1)}{\G(a+1)\G(s+1)}.
\]
Since the digamma function is increasing and \(a\le0\),
\[
\frac{d}{ds}\log c_a(s)
=\psi(s+a+1)-\psi(s+1)\le0.
\]
Thus \(0<c_a(s)\le c_a(0)=1\), and endpoint evaluation gives
\begin{equation}\label{eq:negative-right-endpoint}
Y(1)^2=c_a(n)c_a(n+d)\le1.
\end{equation}
If \(n=0\), the function is a positive constant times
\(((1+x)/2)^{d/2}\), so the endpoint estimate proves the result.
Henceforth let \(n\ge1\).

Suppose first that \(d>0\).  Put
\[
\lambda=n(n+a+d+1),\qquad
\Lambda=\lambda+\frac{d(a+1)}2+\frac{d^2}{4}.
\]
Conjugating the Jacobi equation by \((1+x)^{d/2}\) gives
\begin{equation}\label{eq:negative-conjugated-ode}
(1-x^2)Y''-[a+(a+2)x]Y'
+\left(\Lambda-\frac{d^2}{2(1+x)}\right)Y=0.
\end{equation}
Set
\[
p(x)=(1-x)^{a+1}(1+x),\qquad
V(x)=\Lambda-\frac{d^2}{2(1+x)},\qquad
q(x)=\frac{p(x)V(x)}{1-x^2}.
\]
Then \((pY')'+qY=0\), and
\begin{equation}\label{eq:negative-sturm-product}
\Pi(x):=p(x)q(x)
=(1-x)^{2a+1}
\left[\Lambda(1+x)-\frac{d^2}{2}\right].
\end{equation}
The bracket vanishes at
\[
x_0=-1+\frac{d^2}{2\Lambda}\in(-1,1),
\]
where the inclusion follows from
\[
4\Lambda-d^2
=4n(n+a+d+1)+2d(a+1)>0.
\]
Moreover, \(q>0\) on \((x_0,1)\), and
\begin{equation}\label{eq:negative-sturm-derivative}
\Pi'(x)=(1-x)^{2a}G(x),
\end{equation}
where
\[
G(x)=\Lambda(1-x)
-(2a+1)\left[\Lambda(1+x)-\frac{d^2}{2}\right].
\]
Here
\[
G'(x)=-2(a+1)\Lambda<0,
\qquad G(x_0)=\Lambda(1-x_0)>0.
\]
Thus \(\Pi'\) is either nonnegative throughout \((x_0,1)\), or it
changes sign exactly once, from positive to negative.

Choose the sign so that \(Y>0\) near \(-1\).  Since
\[
P_n^{(a,d)}(-1)
=(-1)^n\frac{\G(n+d+1)}{n!\,\G(d+1)}\ne0,
\]
one has, after this choice,
\[
Y(x)=C_0(1+x)^{d/2}\bigl(1+O(1+x)\bigr),
\qquad C_0>0,
\]
as \(x\downarrow-1\).  Hence, on \((-1,x_0)\), where \(q\le0\),
\[
p(x)Y'(x)=C_1(1+x)^{d/2}(1+o(1))\longrightarrow0,
\qquad C_1>0,\quad x\downarrow-1.
\]
As long as \(Y>0\), \((pY')'=-qY\ge0\); therefore \(Y'>0\), so
\(Y\) stays positive and increasing through \(x_0\).  Consequently,
the first-lobe maximum lies in the region \(q>0\).

On \((x_0,1)\), introduce the Sonin function
\begin{equation}\label{eq:negative-sonin}
S(x)=Y(x)^2+\frac{p(x)}{q(x)}Y'(x)^2
=Y(x)^2+\frac{1-x^2}{V(x)}Y'(x)^2.
\end{equation}
As in \eqref{eq:sonin-derivative},
\begin{equation}\label{eq:negative-sonin-derivative}
S'(x)=-\Pi'(x)\left(\frac{Y'(x)}{q(x)}\right)^2.
\end{equation}
At a nonzero local maximum of \(Y^2\) (which we call a peak), one has
\(Y'=0\) and hence
\(S=Y^2\).  If \(\Pi'\ge0\) throughout \((x_0,1)\), then \(S\) is decreasing
and every peak after the first one is bounded by \(A\).  Otherwise let
\(\xi\) be its unique sign-change point.  If \(\xi\) lies after the
first peak, monotonic decrease up to \(\xi\) bounds all intervening
peaks by \(A\); if \(\xi\) lies before the first peak, there are no
earlier peaks to consider.  On \((\xi,1)\), \(S\) is increasing.
Moreover,
\[
V(1)=\lambda+\frac{d(a+1)}2>0,
\]
and \(Y'\) is finite at \(x=1\).  Hence
\[
\lim_{x\uparrow1}S(x)=Y(1)^2,
\]
so every peak in \((\xi,1)\) is bounded by \(Y(1)^2\).
Equations \eqref{eq:negative-first-peak} and
\eqref{eq:negative-right-endpoint}
prove \(|Y|\le1\) when \(d>0\).

If \(d=0\), then \(Y=P_n^{(a,0)}\), \(Y(-1)=(-1)^n\), and
\eqref{eq:negative-right-endpoint} still holds.  The same Sonin
calculation applies with \(x_0=-1\) and \(V=\Lambda=\lambda>0\) on
the entire interval.  Here \(\Pi'\) has at most one change from positive
to negative, so its Sonin function decreases and then, possibly,
increases.  Its derivative term vanishes at both endpoints, giving
\[
S(-1)=Y(-1)^2=1,
\qquad S(1)=Y(1)^2\le1.
\]
Every interior maximum is therefore bounded by one.  This finishes the
proof.
\end{proof}

\section{Endpoint-to-bulk bounds}\label{sec:endpoint-bulk}

The KKT estimate controls the hard edges but, for fixed parameters,
does not decay with the degree.  The independent degree--parameter
estimate of Bai and Li
\cite[Theorem~1.1]{BaiLiEMNKrasikov2026} supplies the complementary
bulk input anticipated in \cite[Remark~6.4(ii)]{KKT2018}, with the
additional degree--parameter refinement.  Combining the two results
gives the endpoint-to-bulk envelope below.

The companion theorem is used only in \Cref{cor:endpoint-bulk} and in
the second branch of \Cref{thm:two-scale-kernel}.  Darboux asymptotics
show separately that the resulting \(n^{-1/2}\) order is attained on
fixed diagonals.  The contractivity theorems, the temporal phase law,
the scaling limits and their lower bounds, and the nonlinear
applications do not require the companion estimate.

Let \(p_n^{(a,b)}=P_n^{(a,b)}/\sqrt{h_n^{a,b}}\) be orthonormal for
the Jacobi weight \((1-x)^a(1+x)^b\), where
\begin{equation}\label{eq:jacobi-square-norm}
h_n^{a,b}
=\frac{2^{a+b+1}}{2n+a+b+1}
\frac{\G(n+a+1)\G(n+b+1)}
{\G(n+1)\G(n+a+b+1)}.
\end{equation}
This is the standard Jacobi norm formula; see
\cite[Equation~(4.3.3)]{Szego1975}.
Define
\begin{equation}\label{eq:F-def}
F_{n,a,b}(x)
=(1-x^2)^{1/4}(1-x)^{a/2}(1+x)^{b/2}
|p_n^{(a,b)}(x)|.
\end{equation}
For \(k\in\N_0\) and \(S\ge0\), set
\begin{equation}\label{eq:companion-envelope-def}
\mathcal E_k(S)
=\max\left\{1,S^{1/6},
\frac{S^{1/4}}{(k+1)^{1/12}}\right\}.
\end{equation}
The normalization in \eqref{eq:def-g} and
\eqref{eq:jacobi-square-norm} give the exact all-degree identity
\begin{equation}\label{eq:F-g-bridge}
F_{n,a,b}(x)
=\left(\frac{2n+a+b+1}{2}\right)^{1/2}
(1-x^2)^{1/4}|g_n^{(a,b)}(x)|.
\end{equation}
The companion theorem states that, with
\(C_*=1.2\cdot10^5\),
\begin{equation}\label{eq:companion-envelope}
F_{n,a,b}(x)\le C_*\mathcal E_n(a+b+1),
\qquad n\in\N_0,\quad a,b\ge0,\quad -1\le x\le1.
\end{equation}

\begin{corollary}[Endpoint-to-bulk envelope]
\label{cor:endpoint-bulk}
For \(n\in\N_0\), \(a,b\ge0\), and \(-1<x<1\), put
\(D=2n+a+b+1\).  Then
\begin{equation}\label{eq:endpoint-bulk-envelope}
|g_n^{(a,b)}(x)|
\le
\min\left\{
T_{n,a,b},
\frac{\sqrt2\,C_*\mathcal E_n(a+b+1)}
{D^{1/2}(1-x^2)^{1/4}}
\right\}.
\end{equation}
At \(x=\pm1\), the first branch remains valid.
\end{corollary}

\begin{proof}
The first branch is \Cref{thm:main}.  Combining
\eqref{eq:F-g-bridge} with the companion estimate
\eqref{eq:companion-envelope} gives the second branch.
\end{proof}

\begin{remark}[Complementary scales]\label{rem:complementary-scales}
The first term in \eqref{eq:endpoint-bulk-envelope} stays finite at the
hard edges but carries no decay in \(n\) when \(a,b\) are fixed.  The
second has the bulk order \(D^{-1/2}\), but its endpoint factor
diverges.  Thus neither branch subsumes the other.  For fixed \(a,b\),
one has \(T_{n,a,b}\to1\) and \(\mathcal E_n(a+b+1)=O_{a,b}(1)\), so
the two terms exchange dominance at the natural hard-edge scale
\(1-x^2\asymp D^{-2}\).  When the parameters vary with the degree,
the factor \(\mathcal E_n\) retains the sharper Krasikov scale from
the companion theorem.
\end{remark}

\section{Laguerre and Bessel consequences}\label{sec:laguerre}

\subsection{Laguerre functions}

The Jacobi--Laguerre limit of \Cref{thm:main} gives a global Laguerre
bound.  A direct inverse-moment argument gives a second estimate for
the first lobe.  We use the
standard normalization
\[
\mathcal L_n^{(a)}(u)=
\left[\frac{n!}{\G(n+a+1)}\right]^{1/2}
u^{a/2}\e^{-u/2}L_n^{(a)}(u),
\qquad u\ge0,
\]
with endpoint values understood by continuous extension.

Pointwise and extremal inequalities for Laguerre functions have a
substantial literature; see, for example,
\cite{Love1997,KrasikovLaguerre2007,Fejzullahu2021}.  The two results
below retain the finite-parameter normalization and the exact inverse
moment arising from the Jacobi problem.

\begin{corollary}[Global Laguerre bound]\label{cor:laguerre-limit}
For every \(n\in\N_0\), \(a\ge0\), and \(u\ge0\),
\begin{equation}\label{eq:laguerre-bound}
\left|\mathcal L_n^{(a)}(u)\right|
\le
\left(\frac{n+1}{n+a+1}\right)^{1/4}.
\end{equation}
\end{corollary}

\begin{proof}
Fix \(n,a,u\), and in \Cref{thm:main} set
\[
\alpha=a,\qquad \beta=\lambda,
\qquad x_\lambda=1-\frac{2u}{\lambda},
\]
where \(\lambda>u\).  The standard Jacobi--Laguerre limit
\cite[Section~8.22]{Szego1975} gives
\[
P_n^{(a,\lambda)}\left(1-\frac{2u}{\lambda}\right)
\longrightarrow L_n^{(a)}(u).
\]
Moreover,
\[
\left(1-\frac{u}{\lambda}\right)^{\lambda/2}
\longrightarrow\e^{-u/2},
\qquad
\lambda^{-a}
\frac{\G(n+a+\lambda+1)}{\G(n+\lambda+1)}
\longrightarrow1.
\]
It follows from \eqref{eq:def-g} that
\[
g_n^{(a,\lambda)}(x_\lambda)
\longrightarrow \mathcal L_n^{(a)}(u),
\]
whereas
\[
T_{n,a,\lambda}
\longrightarrow
\left(\frac{n+1}{n+a+1}\right)^{1/4}.
\]
Passing to the limit in \Cref{thm:main} proves the assertion.
\end{proof}

The limit preserves both the finite-parameter form and the KKT factor.
For integral \(a\), the constant-one consequence is compatible with
the matrix-coefficient interpretation.  We have not found the sharper
fractional factor \(((n+1)/(n+a+1))^{1/4}\) in the literature for
general real \(a\ge0\); in particular, this finite-parameter form
appears to be new for nonintegral \(a\).  It complements, rather than
replaces, the classical degree-asymptotic Laguerre theory; see
\cite[Chapter~8]{Szego1975}.

The inverse-moment argument for the first Jacobi lobe also has a direct
Laguerre counterpart.  It is independent of
\Cref{cor:laguerre-limit}.

\begin{proposition}[Laguerre inverse moment and first-lobe bound]
\label{prop:laguerre-lobe}
Let \(n\in\N_0\) and \(a>0\).  If \(n\ge1\), let \(\zeta_1\) be the
smallest zero of \(L_n^{(a)}\); if \(n=0\), put
\(\zeta_1=\infty\).  The function \(\mathcal L_n^{(a)}\) has a unique
maximum point \(u_1\) on \((0,\zeta_1)\).  Writing
\(A=\mathcal L_n^{(a)}(u_1)\), one has
\begin{equation}\label{eq:laguerre-inverse-moment}
\int_0^\infty\frac{\mathcal L_n^{(a)}(u)^2}{u}\,du=\frac1a
\end{equation}
and, for \(0\le t\le1+2/a\),
\begin{equation}\label{eq:laguerre-envelope}
\frac{\mathcal L_n^{(a)}(tu_1)}{A}
\ge t^{a/2}\left(1+\frac a2(1-t)\right).
\end{equation}
Consequently,
\begin{equation}\label{eq:laguerre-first-mass}
\int_0^{u_1}\frac{\mathcal L_n^{(a)}(u)^2}{u}\,du
\ge\frac{A^2}{a}R(a),
\qquad A^2<\frac1{R_*(a)},
\end{equation}
where \(R\) and \(R_*\) are defined in
\eqref{eq:R-def} and \eqref{eq:Rstar-def}.
\end{proposition}

\begin{proof}
The contiguous relation
\[
L_j^{(a)}-L_{j-1}^{(a)}=L_j^{(a-1)},
\qquad L_{-1}^{(a)}:=0,
\]
gives \(L_n^{(a)}=\sum_{j=0}^nL_j^{(a-1)}\).  Since
\(a-1>-1\), orthogonality with parameter \(a-1\) yields
\[
\int_0^\infty u^{a-1}\e^{-u}
\bigl(L_n^{(a)}(u)\bigr)^2\,du
=\sum_{j=0}^n\frac{\G(j+a)}{j!}.
\]
If \(C_j=\G(j+a+1)/j!\) and \(C_{-1}=0\), then
\[
\frac{\G(j+a)}{j!}=\frac{C_j-C_{j-1}}a.
\]
The sum telescopes to \(\G(n+a+1)/(a n!)\), and multiplication by
the normalization factor proves \eqref{eq:laguerre-inverse-moment}.

All zeros \(0<\zeta_1<\cdots<\zeta_n\) of \(L_n^{(a)}\) are simple.
On the first lobe,
\[
\mathcal L_n^{(a)}(u)=
C u^{a/2}\e^{-u/2}\prod_{j=1}^n
\left(1-\frac{u}{\zeta_j}\right),
\qquad C>0,
\]
where the product is empty for \(n=0\).  Moreover,
\Cref{thm:canonical-product-lobe} applies in its infinite-endpoint
case with \(\lambda=1\).  It gives the uniqueness of \(u_1\), the
envelope \eqref{eq:laguerre-envelope}, and the first inequality in
\eqref{eq:laguerre-first-mass}.  Comparison with
\eqref{eq:laguerre-inverse-moment} and the full-moment estimate in
\Cref{thm:canonical-product-lobe} give \(A^2<R_*(a)^{-1}\).
\end{proof}

\begin{remark}[Endpoint and sharpness]\label{rem:laguerre-edge}
The restriction \(a>0\) in \eqref{eq:laguerre-inverse-moment} is
sharp: its integrand is asymptotic to a positive constant times
\(u^{a-1}\) at zero.  When \(a=0\), the initial lobe decreases from
\(\mathcal L_n^{(0)}(0)=1\), so its peak is
\(A=1=R_*(0)^{-1/2}\), with \(R_*(0):=1\) by continuity, although the
inverse moment diverges.  Indeed, on
that lobe the zero product gives
\[
(\log \mathcal L_n^{(0)})'(u)
=-\frac12-\sum_j\frac1{\zeta_j-u}<0.
\]
For each
fixed \(n\), let \(A_{n,a}\) denote the first-lobe peak.  If \(n\ge1\),
continuity of the smallest zero at \(a=0\) shows that \(u=a\) lies in
the first lobe for all sufficiently small \(a>0\); for \(n=0\) there
is no zero.  Moreover,
\[
\mathcal L_n^{(a)}(a)=
\left[\frac{n!}{\G(n+a+1)}\right]^{1/2}
a^{a/2}\e^{-a/2}L_n^{(a)}(a)
\longrightarrow L_n^{(0)}(0)=1,
\qquad a\downarrow0.
\]
Since
\[
\mathcal L_n^{(a)}(a)\le A_{n,a}<R_*(a)^{-1/2},
\]
\Cref{prop:laguerre-lobe} gives
\[
A_{n,a}\longrightarrow1,
\qquad A_{n,a}^2R_*(a)\longrightarrow1
\quad\text{as }a\downarrow0.
\]
Thus the universal first-lobe estimate is asymptotically sharp at the
parameter edge, while equality is impossible for every \(a>0\).
\end{remark}

\subsection{A countable-product Bessel application}
\label{sec:bessel-application}

The Jacobi and Laguerre applications above involve finite zero
products.  Applying the countable clause of
\Cref{thm:canonical-product-lobe} to the canonical product of
\(J_\nu(\sqrt u)\) gives the following first-peak estimate.  It is not
used in the discrete Laguerre application.

\begin{corollary}[Bessel first peak]
\label{cor:bessel-first-peak}
Let \(\nu>0\), and let \(j'_{\nu,1}\) be the first positive zero of
\(J_\nu'\).  Then
\begin{equation}\label{eq:bessel-global-maximum}
\sup_{x\ge0}|J_\nu(x)|
=J_\nu(j'_{\nu,1})
<\frac1{\sqrt{R_*(\nu)}}
=\left[
\frac{2(\nu+1)}{\nu+2}
\left(\frac{\nu}{\nu+2}\right)^\nu
\right]^{1/2}.
\end{equation}
For \(\nu=0\), the corresponding non-strict inequality is an equality:
\[
\sup_{x\ge0}|J_0(x)|=1=R_*(0)^{-1/2}.
\]
\end{corollary}

\begin{proof}
The classical canonical product and the zero asymptotic
\(j_{\nu,k}=\pi(k+\nu/2-1/4)+O(k^{-1})\)
\cite[Equations~10.21.15 and 10.21.19]{DLMF} give
\[
J_\nu(z)=\frac{(z/2)^\nu}{\G(\nu+1)}
\prod_{k=1}^{\infty}\left(1-\frac{z^2}{j_{\nu,k}^2}\right),
\qquad \sum_{k=1}^\infty j_{\nu,k}^{-2}<\infty.
\]
Thus \(\Phi_\nu(u)=J_\nu(\sqrt u)\) satisfies the countable-product
case of \Cref{thm:canonical-product-lobe} with \(a=\nu\) and
\(\lambda=0\).  Its first zero is \(\zeta_1=j_{\nu,1}^2\), while
\(\Phi_\nu'(u)=J_\nu'(\sqrt u)/(2\sqrt u)\) shows that its first-lobe
maximum occurs at \(u_1=(j'_{\nu,1})^2\).  By the
Weber--Schafheitlin formula
\cite[Eq.~10.22.57]{DLMF},
\begin{equation}\label{eq:bessel-inverse-moment}
\int_0^\infty\frac{\Phi_\nu(u)^2}{u}\,du
=2\int_0^\infty\frac{J_\nu(x)^2}{x}\,dx
=\frac1\nu.
\end{equation}
The full-moment estimate in the canonical-product theorem therefore
bounds the first peak by \(R_*(\nu)^{-1/2}\).

For completeness, this first peak is the global one.  Bessel's
equation has the self-adjoint form
\[
(xy')'+\left(x-\frac{\nu^2}{x}\right)y=0,
\qquad y=J_\nu.
\]
Near zero, \(y,y'>0\) and \(xy'\to0\).  As long as
\(0<x\le\nu\) and \(y>0\), one has
\((xy')'=(\nu^2/x-x)y\ge0\), so a first-zero argument shows that
\(y\) is positive and increasing there.  Hence \(j'_{\nu,1}>\nu\).
For \(x>\nu\), the Sonin function
\[
\mathcal S(x)=y(x)^2+\frac{x^2}{x^2-\nu^2}y'(x)^2
\]
satisfies
\[
\mathcal S'(x)
=-\frac{2x^3}{(x^2-\nu^2)^2}y'(x)^2\le0.
\]
By the definition of \(j'_{\nu,1}\), one has \(J_\nu'>0\) before this
first critical point; thereafter,
\[
|J_\nu(x)|^2\le\mathcal S(x)
\le\mathcal S(j'_{\nu,1})=J_\nu(j'_{\nu,1})^2,
\qquad x\ge j'_{\nu,1}.
\]
Thus no later peak exceeds the first one.  The case \(\nu=0\) follows
from the integral representation
\[
J_0(x)=\frac1\pi\int_0^\pi\cos(x\sin\theta)\,d\theta,
\]
from \cite[Equation~10.9.1]{DLMF}, which gives
\(|J_0(x)|\le1=J_0(0)\).
\end{proof}

\begin{remark}
The estimate is especially effective at small order.  It is strictly
below the classical unit bound for every \(\nu>0\), and it is
asymptotically sharp as \(\nu\downarrow0\): its right-hand side tends
to one, while \(\sup_x|J_\nu(x)|\to1\).  It is not a large-order
optimization: the bound tends to \(\sqrt2/\e\), whereas Landau's
classical estimate is \(O(\nu^{-1/3})\); see \cite{Landau2000}.  The
inverse-moment argument therefore passes from finite Jacobi and
Laguerre products to the countable Bessel product.
\end{remark}

\section{The discrete Laguerre evolution and its phase transition}
\label{sec:discrete-laguerre}

\subsection{Kernel representation and norm bounds}

We now apply the preceding polynomial estimates to the discrete
Laguerre evolution studied in \cite{KKT2018}.  For
\(\alpha>-1\), the discrete Laguerre operator
\(H_\alpha\) on \(\ell^2(\N_0)\) is the self-adjoint Jacobi operator
associated with
\[
\begin{aligned}
(H_\alpha u)_n={}&
\sqrt{n(n+\alpha)}\,u_{n-1}+(2n+1+\alpha)u_n\\
&+\sqrt{(n+1)(n+1+\alpha)}\,u_{n+1},
\qquad u_{-1}=0.
\end{aligned}
\]
For \(m\ge n\), its exact kernel is related to the weighted Jacobi
function by
\begin{equation}\label{eq:kernel-identity}
\left|\e^{-\ii tH_\alpha}(n,m)\right|
=\frac1{\sqrt{1+t^2}}
\left|g_n^{(\alpha,m-n)}(x_t)\right|,
\qquad x_t=\frac{t^2-1}{t^2+1};
\end{equation}
see \cite[Equation~(6.1)]{KKT2018}.

The one-sided inequality of \Cref{thm:negative-contractivity} closes
the negative-parameter problem raised in \cite[Remark~6.6]{KKT2018}.

\begin{theorem}[Exact norm for negative parameters]
\label{thm:negative-dispersion}
For \(-1<\alpha\le0\) and \(t\in\R\),
\begin{equation}\label{eq:negative-exact-norm}
\left\|\e^{-\ii tH_\alpha}\right\|_{\ell^1\to\ell^\infty}
=(1+t^2)^{-(1+\alpha)/2}.
\end{equation}
In particular, the conjectured decay exponent \(1+\alpha\) is exact.
\end{theorem}

\begin{proof}
For \(m=n+d\), \(d\in\N_0\), use \eqref{eq:kernel-identity}, which is
valid throughout \(\alpha>-1\).  Since
\[
\frac{1-x_t}{2}=\frac1{1+t^2},
\]
the definitions \eqref{eq:def-g} and \eqref{eq:def-B-negative} give
\begin{equation}\label{eq:negative-kernel-factorization}
\left|\e^{-\ii tH_\alpha}(n,n+d)\right|
=(1+t^2)^{-(1+\alpha)/2}
\left|\mathcal B_{n,d}^{(\alpha)}(x_t)\right|.
\end{equation}
For a matrix kernel \(A\) on \(\N_0\),
\[
\|A\|_{\ell^1\to\ell^\infty}
=\sup_{n,m\in\N_0}|A(n,m)|.
\]
Since the kernel of \(\e^{-\ii tH_\alpha}\) is symmetric in \(n,m\),
\Cref{thm:negative-contractivity} and
\eqref{eq:negative-kernel-factorization} give the upper bound in
\eqref{eq:negative-exact-norm}.  At \(n=d=0\), one has
\(\mathcal B_{0,0}^{(\alpha)}\equiv1\), so the \((0,0)\) entry attains
that bound for every \(t\).
\end{proof}

\begin{theorem}[Two-scale kernel estimate]\label{thm:two-scale-kernel}
Let \(\alpha\ge0\), \(n,m\in\N_0\), and \(t\ne0\).  Set
\[
r=\min\{n,m\},\qquad s=\max\{n,m\},\qquad
D=n+m+\alpha+1.
\]
Then
\begin{equation}\label{eq:two-scale-kernel}
\begin{split}
\left|\e^{-\ii tH_\alpha}(n,m)\right|
\le\min\Bigg\{&
\frac1{\sqrt{1+t^2}}
\left[
\frac{(r+1)(s+\alpha+1)}
{(r+\alpha+1)(s+1)}
\right]^{1/4},\\
&\frac{C_*}{|t|^{1/2}}
\frac{\mathcal E_r(\alpha+|m-n|+1)}{D^{1/2}}
\Bigg\}.
\end{split}
\end{equation}
The first branch is also valid at \(t=0\).
\end{theorem}

\begin{proof}
Suppose first that \(m\ge n\).  The first branch follows from
\eqref{eq:kernel-identity} and \Cref{thm:main}.  For the second, apply
\eqref{eq:F-g-bridge} and \eqref{eq:companion-envelope} with
\[
a=\alpha,\qquad b=m-n,\qquad
2n+a+b+1=n+m+\alpha+1=D.
\]
Since
\begin{equation}\label{eq:xt-factor}
(1-x_t^2)^{1/4}
=\frac{\sqrt2\,|t|^{1/2}}{\sqrt{1+t^2}},
\end{equation}
the factor \(\sqrt2\) in \eqref{eq:endpoint-bulk-envelope} cancels
exactly, and the second branch follows.  Since \(H_\alpha\) is a real
symmetric Jacobi matrix, its evolution kernel is symmetric in
\(n,m\), which treats \(m<n\).
\end{proof}

\begin{corollary}[Dispersive and interpolated estimates]
\label{cor:dispersion}
Let \(\alpha\ge0\).  For all \(t\in\R\),
\begin{equation}\label{eq:dispersion}
\left\|\e^{-\ii tH_\alpha}\right\|_{\ell^1\to\ell^\infty}
\le\frac1{\sqrt{1+t^2}}.
\end{equation}
For \(\alpha=0\), equality holds in \eqref{eq:dispersion}.  More
generally, if \(1\le p\le2\) and \(p'=p/(p-1)\), then
\begin{equation}\label{eq:interpolated-dispersion}
\left\|\e^{-\ii tH_\alpha}\right\|_{\ell^p\to\ell^{p'}}
\le(1+t^2)^{-(1/p-1/2)}.
\end{equation}
\end{corollary}

\begin{proof}
The fourth-root factor in the first branch of
\eqref{eq:two-scale-kernel} is at most one, because, when \(s\ge r\),
\[
(r+\alpha+1)(s+1)
-(r+1)(s+\alpha+1)=\alpha(s-r)\ge0.
\]
The supremum of the matrix entries is the
\(\ell^1\)-to-\(\ell^\infty\) norm, which proves
\eqref{eq:dispersion}.  When \(\alpha=0\), the \((0,0)\) entry has
absolute value \((1+t^2)^{-1/2}\); see also
\cite[Theorem~6.1]{KKT2018}.  Finally, interpolate
\eqref{eq:dispersion} with the unitarity of
\(\e^{-\ii tH_\alpha}\) on \(\ell^2\).
\end{proof}

\subsection{Scaling limits and sharpness}

The two branches of \eqref{eq:two-scale-kernel} become effective in
complementary regions of the index--time space.  The next theorem
identifies three asymptotic charts of the same kernel: two hard-edge
charts governing temporal decay and one oscillatory interior chart
governing fixed-diagonal spatial decay.  Write
\[
K_\alpha(t;n,m)=\left|\e^{-\ii tH_\alpha}(n,m)\right|.
\]
For \(t\ne0\), put
\[
\theta_t=\arccos x_t=2\arctan(|t|^{-1})\in(0,\pi).
\]

\begin{theorem}[Three scaling limits]
\label{thm:three-scaling-limits}
Fix \(\alpha\ge0\).
\begin{enumerate}[label=\textup{(\roman*)}]
\item \emph{Bessel hard edge.}  Fix \(d\in\N_0\).  If
\(n_j\in\N_0\), \(t_j\in\R\setminus\{0\}\), \(|t_j|\to\infty\),
and \(n_j/|t_j|\to c\in(0,\infty)\), then
\begin{equation}\label{eq:bessel-scaling}
g_{n_j}^{(\alpha,d)}(x_{t_j})\longrightarrow J_\alpha(2c),
\qquad
|t_j|K_\alpha(t_j;n_j,n_j+d)\longrightarrow|J_\alpha(2c)|.
\end{equation}
More precisely, let \(|t_j|\to\infty\), and for \(c>0\) set
\(n_j(c)=\lfloor c|t_j|\rfloor\).  For every compact interval
\(I\Subset(0,\infty)\),
\begin{align}
\sup_{c\in I}
\left|g_{n_j(c)}^{(\alpha,d)}(x_{t_j})-J_\alpha(2c)\right|
&\longrightarrow0, \label{eq:uniform-bessel-g}\\
\sup_{c\in I}
\left|
|t_j|K_\alpha(t_j;n_j(c),n_j(c)+d)-|J_\alpha(2c)|
\right|
&\longrightarrow0. \label{eq:uniform-bessel-kernel}
\end{align}

\item \emph{Laguerre confluence.}  Fix \(k\in\N_0\).  If
\(d_j\in\N_0\), \(t_j\in\R\setminus\{0\}\), \(d_j\to\infty\),
\(|t_j|\to\infty\), and
\[
\frac{d_j}{1+t_j^2}\longrightarrow u\in[0,\infty),
\]
then
\begin{equation}\label{eq:laguerre-scaling}
g_k^{(\alpha,d_j)}(x_{t_j})
\longrightarrow\mathcal L_k^{(\alpha)}(u),
\qquad
|t_j|K_\alpha(t_j;k,k+d_j)
\longrightarrow|\mathcal L_k^{(\alpha)}(u)|.
\end{equation}
More precisely,
for every \(R>0\), as integers \(d>R\) tend to infinity,
\begin{equation}\label{eq:uniform-jacobi-laguerre}
\sup_{0\le u\le R}
\left|g_k^{(\alpha,d)}\left(1-\frac{2u}{d}\right)
-\mathcal L_k^{(\alpha)}(u)\right|\longrightarrow0
\qquad(d\to\infty).
\end{equation}

\item \emph{Darboux interior.}  Fix \(d\in\N_0\) and \(t\ne0\), and
put
\begin{equation}\label{eq:theta-phi}
\phi_{\alpha,d,t}
=\frac{\alpha+d+1}{2}\theta_t-\frac{2\alpha+1}{4}\pi.
\end{equation}
Then, as \(n\to\infty\),
\begin{equation}\label{eq:fixed-diagonal-asymptotic}
K_\alpha(t;n,n+d)
=\frac{|\cos(n\theta_t+\phi_{\alpha,d,t})|}
{\sqrt{\pi n|t|}}+O(n^{-3/2}).
\end{equation}
The error is uniform for \(t\) in compact subsets of
\(\R\setminus\{0\}\).
\end{enumerate}
\end{theorem}

\begin{proof}
For the first limit, the gamma normalization in
\(g_n^{(\alpha,d)}\) is \(1+O(n^{-1})\).  The locally uniform
Mehler--Heine formula
\[
n^{-\alpha}P_n^{(\alpha,d)}(\cos(z/n))
\longrightarrow(z/2)^{-\alpha}J_\alpha(z)
\]
and the relations
\[
n_j\theta_{t_j}\longrightarrow2c,
\qquad
n_j\sin(\theta_{t_j}/2)\longrightarrow c,
\qquad
\cos^d(\theta_{t_j}/2)\longrightarrow1
\]
give the first limit in \eqref{eq:bessel-scaling}; see
\cite[Theorem~8.1.1]{Szego1975}.  The kernel limit follows from
\eqref{eq:kernel-identity} and
\(|t_j|/\sqrt{1+t_j^2}\to1\).

If \(c\) ranges over \(I\Subset(0,\infty)\), then
\(n_j(c)\asymp_I |t_j|\), and
\[
n_j(c)\theta_{t_j}\longrightarrow2c,
\qquad
n_j(c)\sin(\theta_{t_j}/2)\longrightarrow c
\]
uniformly on \(I\).  The gamma-ratio estimate is then uniform on \(I\),
as is the locally uniform Mehler--Heine formula.  This proves
\eqref{eq:uniform-bessel-g}; the kernel identity proves
\eqref{eq:uniform-bessel-kernel}.

For the second limit, we first prove the uniform statement
\eqref{eq:uniform-jacobi-laguerre}.  Put \(s=u/d\).  Since \(k\) is
fixed, the terminating representation
\begin{align*}
P_k^{(\alpha,d)}(1-2s)
={}&\frac{(\alpha+1)_k}{k!}
\sum_{\ell=0}^k
\frac{(-k)_\ell(k+\alpha+d+1)_\ell}
{(\alpha+1)_\ell\,\ell!}s^\ell
\end{align*}
contains only finitely many terms, and, for every fixed \(R>0\),
\[
(k+\alpha+d+1)_\ell\left(\frac{u}{d}\right)^\ell
\longrightarrow u^\ell
\]
uniformly for \(0\le u\le R\).  Hence
\[
P_k^{(\alpha,d)}\left(1-\frac{2u}{d}\right)
\longrightarrow L_k^{(\alpha)}(u)
\]
uniformly on \([0,R]\).  Moreover,
\[
d^{-\alpha}
\frac{\G(k+\alpha+d+1)}{\G(k+d+1)}
\longrightarrow1,
\qquad
\left(1-\frac{u}{d}\right)^{d/2}
\longrightarrow\e^{-u/2},
\]
the second convergence being uniform on \([0,R]\).  Since
\[
d^{\alpha/2}\left(\frac{u}{d}\right)^{\alpha/2}=u^{\alpha/2},
\]
these estimates prove \eqref{eq:uniform-jacobi-laguerre}, including
\(u=0\) under the continuous endpoint convention.

For the stated sequential limit, set
\[
u_j=\frac{d_j}{1+t_j^2}.
\]
Then \(u_j\to u\), \(x_{t_j}=1-2u_j/d_j\), and the sequence
\(\{u_j\}\) lies in a compact interval.  Applying
\eqref{eq:uniform-jacobi-laguerre} at \(u_j\), followed by continuity
of \(\mathcal L_k^{(\alpha)}\), gives
\[
g_k^{(\alpha,d_j)}(x_{t_j})
\longrightarrow\mathcal L_k^{(\alpha)}(u).
\]
The kernel limit follows from \eqref{eq:kernel-identity} and
\(|t_j|/\sqrt{1+t_j^2}\to1\).

Finally, fixed-parameter Darboux asymptotics give
\begin{equation}\label{eq:darboux-g}
g_n^{(\alpha,d)}(\cos\theta)
=\sqrt{\frac{2}{\pi n\sin\theta}}
\left[
\cos\left(\left(n+\frac{\alpha+d+1}{2}\right)\theta
-\frac{2\alpha+1}{4}\pi\right)+O(n^{-1})
\right]
\end{equation}
uniformly away from the endpoints; see
\cite[Theorem~8.21.8]{Szego1975}.  Insert \(\theta=\theta_t\) and use
\(\sin\theta_t=2|t|/(1+t^2)\).  Taking absolute values changes the
remainder by at most its magnitude and proves
\eqref{eq:fixed-diagonal-asymptotic}.
\end{proof}

The three parts describe complementary faces of
\eqref{eq:two-scale-kernel}.  Where the limiting profiles are nonzero,
part~\textup{(i)}, and part~\textup{(ii)} with \(u>0\), exhibit entries
of temporal order \(|t|^{-1}\), respectively at
\[
n\asymp |t|,\quad |m-n|=O(1),
\qquad\text{and}\qquad
\min\{n,m\}=O(1),\quad |m-n|\asymp t^2.
\]
The endpoint \(u=0\) in part~\textup{(ii)} also covers intermediate
gaps \(|m-n|=o(t^2)\) that tend to infinity.
Part~\textup{(iii)} exhibits the interior order
\(n^{-1/2}|t|^{-1/2}\) on every fixed diagonal.  Thus the temporal and
fixed-diagonal index exponents are attained in the stated regimes.
The fixed-index hard-edge chart becomes completely explicit in the
lowest row.

\begin{proposition}[Negative-binomial lower bound]
\label{prop:negative-binomial}
Let \(\alpha\ge0\), \(s=t^2\), and
\begin{equation}\label{eq:negative-binomial-mass}
p_k^{(\alpha)}(s)=
\frac{\G(k+\alpha+1)}{\G(\alpha+1)\G(k+1)}
\frac{s^k}{(1+s)^{k+\alpha+1}},
\qquad k\in\N_0.
\end{equation}
At \(s=0\), the convention is
\[
p_0^{(\alpha)}(0)=1,
\qquad
p_k^{(\alpha)}(0)=0\quad(k\ge1).
\]
Then
\begin{equation}\label{eq:row-zero-exact}
K_\alpha(t;0,k)^2=p_k^{(\alpha)}(s),
\qquad
\left\|\e^{-\ii tH_\alpha}\right\|_{\ell^1\to\ell^\infty}^2
\ge\max_{k\ge0}p_k^{(\alpha)}(s).
\end{equation}
If \(\alpha>0\), a mode is \(k=\lfloor\alpha s\rfloor\); when
\(\alpha s\) is a positive integer, the two adjacent values
\(k=\alpha s-1,\alpha s\) tie.  Consequently,
\begin{equation}\label{eq:negative-binomial-asymptotic}
\lim_{|t|\to\infty}|t|\max_{k\ge0}K_\alpha(t;0,k)
=A_\alpha,
\qquad
A_\alpha=
\left(\frac{\alpha^\alpha\e^{-\alpha}}{\G(\alpha+1)}\right)^{1/2}.
\end{equation}
For \(\alpha=0\), the same assertion holds with \(A_0=1\).
\end{proposition}

\begin{proof}
The kernel identity at \(n=0\) and the endpoint value
\(P_0^{(\alpha,k)}=1\) give \eqref{eq:row-zero-exact}.  For \(s>0\),
\[
\frac{p_{k+1}^{(\alpha)}(s)}{p_k^{(\alpha)}(s)}
=\frac{s(k+\alpha+1)}{(1+s)(k+1)},
\]
which gives the stated mode.  For \(\alpha>0\), choose the mode
\(k_s=\lfloor\alpha s\rfloor\), or either adjacent mode in the tie
case.  Then \(k_s/s\to\alpha\), and the gamma-ratio asymptotic together
with \((1+1/s)^{-k_s}\to\e^{-\alpha}\) gives
\[
s\,p_{k_s}^{(\alpha)}(s)
\longrightarrow
\frac{\alpha^\alpha\e^{-\alpha}}{\G(\alpha+1)}.
\]
Taking square roots proves
\eqref{eq:negative-binomial-asymptotic}.
For \(\alpha=0\), the mode is zero and
\(p_0^{(0)}(s)=(1+s)^{-1}\).
\end{proof}

For \(\alpha>0\), the maximizing gap satisfies
\(k_s/(1+t^2)\to\alpha\).  Thus the constant in
\eqref{eq:negative-binomial-asymptotic} is exactly the \(k=0\),
\(u=\alpha\) instance of the Laguerre chart:
\[
\mathcal L_0^{(\alpha)}(\alpha)
=\left(\frac{\alpha^\alpha\e^{-\alpha}}
{\G(\alpha+1)}\right)^{1/2}
=A_\alpha.
\]
The negative-binomial identity therefore realizes the second hard-edge
profile exactly within the lowest row.  It supplies a lower bound for
the full kernel norm, not an identification of that norm.

\begin{conjecture}[Lowest-row dominance]
\label{conj:lowest-row}
For every \(\alpha>0\) and \(t\in\R\),
\begin{equation}\label{eq:lowest-row-conjecture}
\left\|\e^{-\ii tH_\alpha}\right\|_{\ell^1\to\ell^\infty}^2
=\max_{k\ge0}p_k^{(\alpha)}(t^2).
\end{equation}
For \(t\ne0\), this has the following equivalent Meixner form.  Set
\(\beta=\alpha+1\), \(c=t^2/(1+t^2)\in(0,1)\), and
\[
w_m=(1-c)^\beta\frac{(\beta)_m}{m!}c^m,
\qquad
h_n=\frac{n!}{(\beta)_n c^n}.
\]
Here \(\{w_m\}_{m\ge0}\) is the negative-binomial probability mass,
and \(h_n\) is the squared norm of \(M_n(\,\cdot\,;\beta,c)\) with
respect to this mass.  The Meixner representation
\cite[Remark~3.2 and Equation~(3.6)]{KKT2018} therefore takes the
normalized form
\begin{equation}\label{eq:meixner-coefficient}
K_\alpha(t;n,m)
=\left(\frac{w_m}{h_n}\right)^{1/2}|M_n(m;\beta,c)|
=
\left[w_m\frac{(\beta)_n}{n!}c^n\right]^{1/2}
|M_n(m;\beta,c)|.
\end{equation}
Thus \eqref{eq:lowest-row-conjecture} says that every normalized
Meixner coefficient in \eqref{eq:meixner-coefficient} is at most
\(\max_m\sqrt{w_m}\).  At \(t=0\), the conjecture is the trivial
identity \(\|I\|_{\ell^1\to\ell^\infty}=1\).
\end{conjecture}

The conjecture is a genuine two-index extremal problem: the fixed-index
limits in \Cref{thm:three-scaling-limits} do not control simultaneous
growth of both indices.  In contrast, \Cref{thm:negative-dispersion}
proves the analogous statement for \(-1<\alpha\le0\), where the largest
atom is always the zeroth one.

\begin{corollary}[Optimal temporal decay exponent]
\label{cor:temporal-sharpness}
For \(\alpha\ge0\), let
\[
M_\alpha=\sup_{z\ge0}|J_\alpha(z)|,
\]
and define the global Laguerre envelope
\begin{equation}\label{eq:Lambda-alpha}
\Lambda_\alpha
=\sup_{k\ge0,\,u\ge0}|\mathcal L_k^{(\alpha)}(u)|.
\end{equation}
Then
\begin{equation}\label{eq:temporal-sharpness}
\begin{aligned}
0<A_\alpha\le\Lambda_\alpha
&\le\liminf_{|t|\to\infty}|t|
\left\|\e^{-\ii tH_\alpha}\right\|_{\ell^1\to\ell^\infty}\\
&\le\limsup_{|t|\to\infty}|t|
\left\|\e^{-\ii tH_\alpha}\right\|_{\ell^1\to\ell^\infty}
\le1.
\end{aligned}
\end{equation}
Moreover, \(M_\alpha\le\Lambda_\alpha\).
Thus the temporal exponent \(1\) is optimal for every \(\alpha\ge0\).
\end{corollary}

\begin{proof}
The upper bound is \Cref{cor:dispersion}.  Let
\(\{t_j\}\subset\R\setminus\{0\}\) be an arbitrary sequence with
\(|t_j|\to\infty\).  Fix \(k\in\N_0\) and \(u>0\), and choose
\(d_j=\lfloor u(1+t_j^2)\rfloor\).  The Laguerre scaling in
\Cref{thm:three-scaling-limits}(ii), together with domination of every
matrix entry by the kernel norm, gives
\[
\liminf_{j\to\infty}|t_j|
\left\|\e^{-\ii t_jH_\alpha}\right\|_{\ell^1\to\ell^\infty}
\ge |\mathcal L_k^{(\alpha)}(u)|.
\]
The sequence \(\{t_j\}\) was arbitrary.  Taking the supremum over
\(k\) and \(u>0\), followed by continuity at \(u=0\), proves the
\(\Lambda_\alpha\) lower bound.  The identity
\(\mathcal L_0^{(\alpha)}(\alpha)=A_\alpha\) proves
\(A_\alpha\le\Lambda_\alpha\).

Finally, the Laguerre Mehler--Heine formula
\cite[Equation~18.11.6]{DLMF} and the gamma-ratio asymptotic give
\[
\mathcal L_n^{(\alpha)}\left(\frac{z^2}{4n}\right)
\longrightarrow J_\alpha(z)
\qquad(n\to\infty)
\]
for each \(z\ge0\).  Hence \(M_\alpha\le\Lambda_\alpha\).
\end{proof}

If Conjecture~\ref{conj:lowest-row} holds, then
\eqref{eq:negative-binomial-asymptotic} gives the exact large-time
constant \(A_\alpha\).  In view of \eqref{eq:temporal-sharpness}, it
would therefore force the separate Laguerre extremal identity
\begin{equation}\label{eq:laguerre-ground-state-conjecture}
\Lambda_\alpha=A_\alpha,
\qquad \alpha>0.
\end{equation}
In particular, a necessary consistency condition is
\begin{equation}\label{eq:bessel-nb-consistency}
M_\alpha\le A_\alpha,
\qquad \alpha>0.
\end{equation}
We do not prove \eqref{eq:laguerre-ground-state-conjecture} or the full
range of \eqref{eq:bessel-nb-consistency} here.  Conditionally on
\eqref{eq:laguerre-ground-state-conjecture}, the latter follows from
\(M_\alpha\le\Lambda_\alpha=A_\alpha\).  It can, however, be verified
directly for \(\alpha\ge3\).  Landau's explicit estimate
\(M_\alpha\le b_L\alpha^{-1/3}\), with \(b_L<0.675\), and the standard
upper Stirling bound
\[
\G(\alpha+1)
 <\sqrt{2\pi\alpha}\,(\alpha/\e)^\alpha
   \exp\!\left(\frac{1}{12\alpha}\right)
\]
reduce \(M_\alpha\le A_\alpha\) to
\[
b_L^2\sqrt{2\pi}\exp\!\left(\frac{1}{12\alpha}\right)
 \le \alpha^{1/6}.
\]
At \(\alpha=3\), the left-hand side is less than \(1.175\), whereas
the right-hand side exceeds \(1.200\); thereafter the two sides are,
respectively, decreasing and increasing.  Thus only
\(0<\alpha<3\) remains open in \eqref{eq:bessel-nb-consistency}.

\begin{corollary}[Dispersive phase transition]
\label{cor:phase-transition}
For every \(\alpha>-1\),
\begin{equation}\label{eq:phase-transition}
\left\|\e^{-\ii tH_\alpha}\right\|_{\ell^1\to\ell^\infty}
\asymp_\alpha |t|^{-\min\{1,1+\alpha\}},
\qquad |t|\to\infty.
\end{equation}
For \(-1<\alpha\le0\), the stronger identity
\eqref{eq:negative-exact-norm} holds at every time.
\end{corollary}

\begin{proof}
For \(-1<\alpha\le0\), use \Cref{thm:negative-dispersion}.  For
\(\alpha\ge0\), combine \Cref{cor:dispersion} with
\eqref{eq:temporal-sharpness}.
\end{proof}

\begin{corollary}[Optimal fixed-diagonal decay]
\label{cor:spatial-sharpness}
Fix \(\alpha\ge0\), \(d\in\N_0\), and \(t\ne0\).  Then
\begin{equation}\label{eq:fixed-diagonal-limsup}
\limsup_{n\to\infty}\sqrt n\,K_\alpha(t;n,n+d)
\ge\frac1{\sqrt{2\pi|t|}}>0.
\end{equation}
If \(\theta_t/\pi\) is irrational, the left-hand side equals
\((\pi|t|)^{-1/2}\).  Hence the exponent \(n^{-1/2}\) is optimal on
every fixed diagonal.
\end{corollary}

\begin{proof}
Multiplying \eqref{eq:fixed-diagonal-asymptotic} by \(\sqrt n\) shows
that the desired limit is governed by
\[
\limsup_{n\to\infty}
\left|\cos(n\theta_t+\phi_{\alpha,d,t})\right|.
\]
If \(\theta_t/\pi\) is irrational, the orbit modulo \(\pi\) is dense,
and this limsup equals one.  If
\(\theta_t/\pi=p/q\) in lowest terms, then \(q\ge2\), and the orbit
consists of \(q\) equally spaced points modulo \(\pi\), each repeated
infinitely often.  One of them lies within \(\pi/(2q)\) of zero modulo
\(\pi\).  Therefore
\[
\limsup_{n\to\infty}
\left|\cos(n\theta_t+\phi_{\alpha,d,t})\right|
\ge \cos\left(\frac{\pi}{2q}\right)
\ge\frac1{\sqrt2}.
\]
Substitution into \eqref{eq:fixed-diagonal-asymptotic} proves both
claims.
\end{proof}

\section{Strichartz estimates and nonlinear scattering}
\label{sec:nonlinear}

\begin{corollary}[Phase-dependent Strichartz estimates]
\label{cor:strichartz}
Let \(\alpha>-1\), put
\(\sigma_\alpha=\min\{1,1+\alpha\}\), and call \((q,r)\)
\(\alpha\)-admissible if
\[
2\le q,r\le\infty,
\qquad \frac1q+\frac{\sigma_\alpha}{r}
=\frac{\sigma_\alpha}{2},
\qquad (q,r,\sigma_\alpha)\ne(2,\infty,1).
\]
For every \(\alpha\)-admissible pair,
\begin{equation}\label{eq:homogeneous-strichartz}
\left\|\e^{-\ii tH_\alpha}f\right\|_{L^q(\R;\ell^r)}
\le C_{\alpha,q,r}\|f\|_{\ell^2}.
\end{equation}
If \((q,r)\) and \((\widetilde q,\widetilde r)\) are
\(\alpha\)-admissible, then
\begin{equation}\label{eq:inhomogeneous-strichartz}
\left\|
\int_{s<t}\e^{-\ii(t-s)H_\alpha}F(s)\,ds
\right\|_{L^q(\R;\ell^r)}
\le C_{\alpha,q,r,\widetilde q,\widetilde r}
\|F\|_{L^{\widetilde q'}(\R;\ell^{\widetilde r'})}.
\end{equation}
When \(\alpha\ge0\), \(\sigma_\alpha=1\), and the constants may be
chosen independently of \(\alpha\).
\end{corollary}

\begin{proof}
Write \(U_\alpha(t)=\e^{-\ii tH_\alpha}\).  Unitarity and the group
property give, for \(t\ne s\),
\[
\|U_\alpha(t)U_\alpha(s)^*\|_{\ell^1\to\ell^\infty}
=\|U_\alpha(t-s)\|_{\ell^1\to\ell^\infty}
\le(1+(t-s)^2)^{-\sigma_\alpha/2}
\le|t-s|^{-\sigma_\alpha},
\]
by \Cref{thm:negative-dispersion,cor:dispersion}.  Apply the
measure-space form of the Keel--Tao theorem \cite{KeelTao1998} with
\(\N_0\) equipped with counting measure and energy space
\(\ell^2(\N_0)\).  If \(\vartheta=1-2/r\), its sharp admissibility
condition is
\[
\frac1q=\frac{\sigma_\alpha\vartheta}{2}
=\frac{\sigma_\alpha}{2}-\frac{\sigma_\alpha}{r},
\]
which is precisely the relation in the statement; the forbidden
Keel--Tao endpoint is \((q,r,\sigma_\alpha)=(2,\infty,1)\).  The
homogeneous and retarded inhomogeneous conclusions are therefore
\eqref{eq:homogeneous-strichartz} and
\eqref{eq:inhomogeneous-strichartz}.

When \(\alpha\ge0\), both the energy constant and the dispersive
constant are one and \(\sigma_\alpha=1\).  Hence the abstract constants
depend only on the displayed exponent pairs, not on \(\alpha\).
\end{proof}

We record one standard nonlinear consequence of the phase-dependent
admissible line.  The power below is selected by the diagonal
Strichartz relation; no scaling symmetry or sharp large-data threshold
is asserted.

\begin{corollary}[Small-data scattering at the decay-determined power]
\label{cor:nonlinear-scattering}
Let \(\alpha>-1\), set
\[
\sigma_\alpha=\min\{1,1+\alpha\},\qquad
p_\alpha=1+\frac{2}{\sigma_\alpha},\qquad
Q_\alpha=p_\alpha+1,
\]
and let \(\mu\in\R\).  There is
\(\varepsilon_{\alpha,\mu}>0\) with the following property.  If
\(u_0\in\ell^2(\N_0)\) and
\(\|u_0\|_{\ell^2}\le\varepsilon_{\alpha,\mu}\), then
\begin{equation}\label{eq:critical-dnls}
 \ii\partial_tu=H_\alpha u+\mu|u|^{p_\alpha-1}u,
 \qquad u(0)=u_0,
\end{equation}
has a unique global mild solution \(u\) in
\begin{equation}\label{eq:critical-scattering-space}
 C(\R;\ell^2)\cap
 L^{Q_\alpha}(\R;\ell^{Q_\alpha}).
\end{equation}
Moreover, there are \(u_\pm\in\ell^2\) such that
\begin{equation}\label{eq:critical-scattering}
 \|u(t)-\e^{-\ii tH_\alpha}u_\pm\|_{\ell^2}
 \longrightarrow0\qquad(t\to\pm\infty).
\end{equation}
For every \(\mu_0\ge0\), there is an \(\varepsilon(\mu_0)>0\),
independent of \(\alpha\ge0\) and of \(\mu\) with
\(|\mu|\le\mu_0\), for which the same conclusion holds whenever
\(\|u_0\|_{\ell^2}\le\varepsilon(\mu_0)\).
\end{corollary}

\begin{proof}
Write \(\sigma=\sigma_\alpha\), \(p=p_\alpha\), and \(Q=Q_\alpha\).
Since \(0<\sigma\le1\),
\[
 Q=p+1=2+\frac{2}{\sigma}
 =\frac{2(1+\sigma)}{\sigma}\ge4,
 \qquad
 \frac1Q+\frac{\sigma}{Q}=\frac{\sigma}{2},
\]
the pair \((Q,Q)\) is \(\alpha\)-admissible and is never the forbidden
endpoint.  On either half-line
\(I=\R_+\) or \(I=\R_-\), use the space
\[
 X(I)=L^\infty(I;\ell^2)\cap L^Q(I;\ell^Q)
\]
and the Duhamel map
\[
 \mathcal Tu(t)=\e^{-\ii tH_\alpha}u_0
 -\ii\mu\int_0^t\e^{-\ii(t-s)H_\alpha}
 |u(s)|^{p-1}u(s)\,ds.
\]
Because \(Q=p+1\), H\"older's inequality on
\(I\times\N_0\) gives the exact exponent relation
\begin{equation}\label{eq:critical-nonlinearity-bound}
 \bigl\||u|^{p-1}u\bigr\|_{L^{Q'}(I;\ell^{Q'})}
 =\|u\|_{L^Q(I;\ell^Q)}^p.
\end{equation}
The homogeneous and inhomogeneous estimates in
\Cref{cor:strichartz}, used with the admissible targets
\((\infty,2)\) and \((Q,Q)\), therefore imply
\[
 \|\mathcal Tu\|_{X(I)}
 \le C_\alpha\|u_0\|_{\ell^2}
 +C_\alpha|\mu|\|u\|_{X(I)}^p.
\]
Here functions on \(I\) are extended by zero; on \(\R_-\) we use the
advanced counterpart of \eqref{eq:inhomogeneous-strichartz}, which
follows from the retarded estimate by taking adjoints, equivalently by
time reversal.
The pointwise Lipschitz estimate for
\(z\mapsto|z|^{p-1}z\), followed by H\"older, similarly gives
\[
 \|\mathcal Tu-\mathcal Tv\|_{X(I)}
 \le C_{\alpha,p}|\mu|
 \bigl(\|u\|_{X(I)}^{p-1}+\|v\|_{X(I)}^{p-1}\bigr)
 \|u-v\|_{X(I)}.
\]
For sufficiently small \(\|u_0\|_{\ell^2}\), the map is a
contraction on a ball in \(X(I)\).  Applying this on both half-lines
gives the global solution; the standard Duhamel argument also gives its
\(C_t\ell^2\) continuity.  To obtain uniqueness in the whole class
\eqref{eq:critical-scattering-space}, let \(u,v\) be two such mild
solutions with the same initial value.  Absolute continuity of their
\(L^Q_t\ell^Q\) norms partitions each half-line into finitely many
successive intervals on which
\(\|u\|_{L^Q\ell^Q}^{p-1}+\|v\|_{L^Q\ell^Q}^{p-1}\) is smaller than
the reciprocal of the constant in the difference estimate.  Starting
at the origin and iterating that estimate on these intervals gives
\(u=v\).

By the dual homogeneous Strichartz estimate and
\eqref{eq:critical-nonlinearity-bound}, the integrals
\[
 u_\pm=u_0-\ii\mu\int_0^{\pm\infty}
 \e^{\ii sH_\alpha}|u(s)|^{p-1}u(s)\,ds
\]
converge in \(\ell^2\).  Applying the same estimate on the tails of
the time integral proves \eqref{eq:critical-scattering}.  Finally, for
\(\alpha\ge0\) one has \(\sigma=1\), \(p=3\), \(Q=4\); the relevant
Strichartz constants are uniform in \(\alpha\), and the cubic
Nemytskii constant is fixed.  Hence a single threshold
\(\varepsilon(\mu_0)>0\) works uniformly for all \(\alpha\ge0\) and
\(|\mu|\le\mu_0\).
\end{proof}

\section{Discussion and open problems}

The KKT theorem and the negative-parameter contractivity theorem
together determine the optimal decay exponent for every
\(\alpha>-1\).  They are independent of the companion
EMN--Krasikov theorem, which enters only in the degree-sensitive bulk
branch of \Cref{cor:endpoint-bulk} and \Cref{thm:two-scale-kernel}.

The reusable part of the proof is the passage from an exact singular
moment to a first-lobe bound for a real-zero canonical product.  Its
constants yield Laguerre and countable-product Bessel consequences.
It would be useful to identify other
canonical products for which the corresponding moment is explicitly
computable.

The remaining positive-parameter questions concern exact extrema rather
than decay exponents.  Foremost is the lowest-row dominance
conjecture~\ref{conj:lowest-row}, equivalently an extremal problem for
normalized Meixner coefficients.  A necessary confluent consequence of
lowest-row dominance is the separate extremal question whether
\[
\sup_{n\ge0,\,u\ge0}|\mathcal L_n^{(\alpha)}(u)|
=\left(
\frac{\alpha^\alpha\e^{-\alpha}}{\G(\alpha+1)}
\right)^{1/2},
\qquad \alpha>0,
\]
with the supremum attained at \(n=0,u=\alpha\).  A separate
finite-parameter problem is to determine the exact value of
\(\max|g_n^{(\alpha,\beta)}|\) and decide which of the two extreme
lobes wins.  \Cref{thm:intro-extreme-lobes} proves that no interior
lobe can maximize, but it does not compare the two extreme peaks in
closed form.  Likewise, \Cref{thm:main} asserts uniform sharpness of
the constant one, not optimality of \(T_{n,\alpha,\beta}\) for every
positive parameter triple.

\subsection*{Declaration of competing interest}

The author declares no competing interests.

\subsection*{Data availability}

No data were used or generated in this study.

\subsection*{Declaration of generative AI and AI-assisted technologies
in the manuscript preparation process}

During the preparation of this work, the author used OpenAI's ChatGPT
and Codex as research and editorial aids for checking algebraic identities and improving the
exposition.  The author independently verified the mathematical
arguments, reviewed and edited the manuscript, and takes full
responsibility for its content.

\end{document}